\documentclass[10pt,letterpaper]{amsart}
\usepackage[utf8]{inputenc}
\usepackage{amsmath}
\usepackage{amsfonts}
\usepackage{amssymb}
\usepackage{enumitem}
\usepackage{mathtools}
\usepackage{amsthm}
\usepackage{csquotes}
\usepackage{hyperref}
\usepackage{xcolor}
\usepackage[backend=biber,
    isbn=false,
    doi=false,
    url=false,
    maxbibnames=99]{biblatex}
\usepackage{mathrsfs}

\newcommand{\matr}[1]{{\begin{pmatrix}#1\end{pmatrix}}}
\newcommand{\smatr}[1]{\big(\!{\begin{smallmatrix}#1\end{smallmatrix}}\!\big)}
\newcommand{\N}{\mathbb{N}}
\newcommand{\Q}{\mathbb{Q}}
\newcommand{\R}{\mathbb{R}}
\newcommand{\SL}{\operatorname{SL}}
\newcommand{\T}{\mathbb{T}}
\newcommand{\Z}{\mathbb{Z}}
\newcommand{\rmpars}[1]{\textnormal(#1\textnormal)}

\renewcommand{\phi}{\varphi}
\DeclareMathOperator{\Img}{im}

\newtheorem{lemma}{Lemma}
\newtheorem{theorem}[lemma]{Theorem}
\newtheorem{corollary}[lemma]{Corollary}

\theoremstyle{definition}
\newtheorem{example}[lemma]{Example}

\title{Discordant sets and ergodic Ramsey theory}

\author{Vitaly Bergelson}
\address{Department of Mathematics, Ohio State University, Columbus, OH 43210, USA}
\email{vitaly@math.ohio-state.edu}
\author{Jake Huryn}
\address{Department of Mathematics, Ohio State University, Columbus, OH 43210, USA}
\email{huryn.5@osu.edu}
\author{Rushil Raghavan}
\address{Department of Mathematics, Ohio State University, Columbus, OH 43210, USA}
\email{raghavan.63@osu.edu}

\subjclass[2020]{05D10, 37A44}

\begin{document}

\maketitle

\begin{abstract}
We explore the properties of non-piecewise syndetic sets with positive upper density, which we call \textit{discordant}, in  countably infinite  amenable (semi)groups.
Sets of this kind are involved in many questions of Ramsey theory and manifest the difference in complexity between the classical van der Waerden's theorem and Szemer\'{e}di's theorem.
We generalize and unify old constructions and obtain new results about these historically interesting sets.
Along the way, we draw from various corners of mathematics, including classical Ramsey theory, ergodic theory, number theory, and topological and symbolic dynamics.
% Here is a small sample of our results.
% \begin{itemize}
% \item
% We connect discordant sets to recurrence in dynamical systems, and in this setting we exhibit an intimate analogy between discordant sets and nowhere dense sets having positive measure.

% \item
% We introduce a wide-ranging generalization of the squarefree numbers, producing many examples of discordant sets in $\Z$, $\Z^d$, and the Heisenberg group.
% We develop a unified method to compute densities of these discordant sets.

% \item
% We show that, for any countable abelian group $G$, any F\o lner sequence $\Phi$ in $G$, and any $c \in (0, 1)$, there exists a discordant set $A \subseteq G$ with $d_\Phi(A) = c$.
% Here $d_\Phi$ denotes density along $\Phi$.
% \end{itemize}
% Along the way, we draw from various corners of mathematics, including classical Ramsey theory, ergodic theory, number theory, and topological and symbolic dynamics.
\end{abstract}

\tableofcontents

\section{Introduction}

Let $A$ be a subset of the set $\N$ of positive integers. We begin by defining some notions of largeness that can be imposed on $A$.
\begin{itemize}
\item $A$ is an \textit{IP set} if it contains a \textit{finite sums set}, i.e., a set of the form \[\operatorname{FS}((a_n)_{n \in \N})={\left\{\sum_{i\in I}a_i\colon\text{$I\subseteq\N$ is finite and nonempty}\right\}}\] for some infinite sequence $(a_n)_{n \in \N}$ of positive integers.
\item $A$ is \textit{thick} if it contains arbitrarily long intervals: for each $\ell\in\N$, there exists $a\in\N$ such that $\{a,a+1,\dots,a+\ell-1\}\subseteq A$.
\item $A$ is \textit{syndetic} if its gaps are bounded: this means that, if we write $A=\{a_1<a_2<\cdots\}$, the successive difference $a_{i+1}-a_i$ is bounded.
Equivalently, $A$ is syndetic if there exist $t_1,\dots,t_r\in\N$ such that
$\N = \bigcup_{i = 1}^r A - t_i$, where $A - t = \{x \in \N \colon x + t \in A\}$.
\item $A$ is \textit{piecewise syndetic} if it can be written as the intersection of a thick set with a syndetic set.
Equivalently, $A$ is piecewise syndetic if there exist $t_1, \dots, t_r\in\N$ for which $\bigcup_{i = 1}^r A - t_i$ is thick.
\item The \textit{upper density of $A$} is \[\overline d(A)=\limsup_{n\to\infty}\frac{|A\cap\{1,\dots,n\}|}{n}.\] When the limit in question exists, this quantity is called the \textit{density of $A$}, denoted by $d(A)$. A set should be thought of as large in terms of (upper) density if this quantity is positive.
%\item More generally, given any sequence $\mathcal{I} = (I_n)$ of intervals in $\N$ of increasing length, we define the upper density of $A$ \textit{with respect to $\mathcal{I}$} to be
%\[
%\overline d_\mathcal{I}(A) = \limsup_{n \to \infty} \frac{|A \cap I_n|}{|I_n|},
%\]
%and analogously we define $d_\mathcal{I}(A)$.
%\item Finally, the \textit{upper Banach density of $A$} is
%\begin{align*}
%d^*(A)&=\limsup_{n-m\to\infty}\frac{|A\cap\{m,m+1,\dots,n\}|}{n-m+1}\\
%&=\sup{\left\{\overline d_\mathcal{I}(A)\colon
%\!\!\begin{array}{c}\text{$\mathcal{I}$ is a sequence of intervals} \\ \text{in $\N$ of increasing length}\end{array}\!\!
%\right\}}.
%\end{align*}
%While upper density measures the proportion of a set along the sequence of intervals $(\{1, \dots, n\})_{n\in \N}$, upper Banach density measures the proportion along the sequence $\mathcal{I}$ of intervals of increasing length which maximizes $\overline d_\mathcal{I}(A)$. (It is easy to show that such a sequence of intervals always exists.)
\end{itemize}

Although we have called each of the above properties a ``notion of largeness'', they are not on equal footing. For all these definitions, $\N$ is large, and if $A$ has a large subset, then $A$ is itself large. However, only some of these notions of largeness are \textit{partition regular}, meaning that if $A$ is large then, for any finite partition $A=\bigcup_{i=1}^r C_i$ into disjoint sets, at least one $C_i$ is large. This third property underlies much of classical Ramsey theory.
The partition regularity of IP sets is a famous theorem of Hindman \cite{Hi}, and the partition regularity of piecewise syndeticity is a classical fact, often rediscovered (see, e.g., \cite[Lemma 2.4]{Hi2} and \cite[Theorem 1.24]{Fur}).\footnote
{A related fact, that if $\N = \bigcup_{i = 1}^rC_i$ then at least one $C_i$ is piecewise syndetic, appears to have been first observed by Brown \cite[Lemma 1]{Br}.}
The partition regularity of positive upper density
is an easy exercise for the reader, as is checking that thickness and syndeticity fail to be partition regular.

One of the earliest partition results of Ramsey theory is van der Waerden's theorem, which says that given any finite partition $\N = \bigcup_{i = 1}^r C_i$, at least one $C_i$ is \textit{AP-rich}, i.e., contains arithmetic progressions of any finite cardinality \cite{vdW}. 
A useful equivalent form of this theorem states that the family of AP-rich sets is partition regular.
Van der Waerden's theorem has a corresponding density version, Szemer\'edi's theorem \cite{Sz}, one of whose many equivalent formulations guarantees that any set having positive upper
%Banach
density is AP-rich.%Other partition results, such as the Hales-Jewett theorem \cite{HJ}, also have density versions \cite{FK}.

In seeking a density version of Hindman's theorem, Erd\H os asked whether $d(A)>0$ is enough to imply that $A$ contains a shift of a finite sums set.
(Of course we cannot ask that $A$ contain a finite sums set; consider, e.g., the set of odd numbers.) This question was answered negatively in unpublished work of E.\ Straus, who showed that, for any $\varepsilon>0$, one can find $A\subseteq\N$ with $d(A)\geq 1-\varepsilon$ such that $A-n$ is not an IP set for any $n\in\Z$. In \cite[p.\ 105]{Er}, Erd\H os writes the following:
\begin{displayquote}
\small{I first hoped that Hindman's theorem on all subsums $\sum\varepsilon_ia_i$ can be generalised
for sets of positive density, but the following example of E. Straus seems to give a counterexample to all such attempts: Let $p_1<p_2<\cdots$ be a set of primes tending to infinity fast, and let $\sum\varepsilon_i<\infty$. Consider the set of integers $a_1<a_2<\cdots$ so that $a_i\not\equiv\alpha\pmod{p_i}$ where $|\alpha|\leq\varepsilon_ip_i$. The density of this sequence is $\prod_i(1-\varepsilon_i)>0$, and
it is not difficult to see that this sequence furnishes the required counterexample.}
\end{displayquote}
See \cite[Theorem 2.20]{BBHS} for more discussion and a detailed presentation of Straus' construction.

%Let us motivate a question similar to that of Erd\H{o}s.
Recall that if $A$ is piecewise syndetic, then a union of finitely many shifts of $A$ is thick.
Since thick sets are obviously AP-rich, it follows from van der Waerden's theorem %on the partition regularity of AP-richness
that one of the shifts of $A$ (and hence $A$ itself) is AP-rich.
This tempts us to ask, in analogy with Erd\H{o}s' question, whether $d(A)>0$ implies that $A$ is piecewise syndetic: a positive answer to this question would mean that we have just proven Szemer\'{e}di's theorem.
However, Straus' construction again provides a counterexample; it has positive density but is not piecewise syndetic.
Indeed, it is an easy exercise to show that any thick set is IP, so any piecewise syndetic set contains a shift of an IP set by Hindman's theorem.
Thus Straus' example cannot be piecewise syndetic.

%rereword first sentence
% In view of Szemer\'edi's strengthening of van der Waerden's theorem and Erd\H{o}s' desire for a density version of Hindman's theorem, it is natural for one to wonder whether there is a connection between density and piecewise syndeticity which could explain the partition regularity of piecewise syndeticity.
% Along these lines one might hope that any set of positive density is piecewise syndetic. Here Straus' construction again provides a counterexample, since if no shift of $A$ is IP, then $A$ is necessarily non-piecewise syndetic.

It is  counterexamples of the above kind we are interested in---those sets which are non-piecewise syndetic but have positive (upper) density. We will call such sets \textit{discordant}. Another classical example of a discordant set is the set of squarefree numbers. However, this set has shifts which are IP \cite[Theorem 2.8]{BR}, and so has a different nature than Straus' counterexample.
More constructions of discordant sets can be found in \cite{BPS} along with applications to Ramsey theory.
A fruitful analogy can be drawn between discordant sets and nowhere dense sets of positive measure, both being small in a ``topological'' sense but large in an ``analytic'' sense. This analogy will be developed and more closely investigated in the following sections.

We will now more closely examine the relationship between van der Waerden's theorem, Szemer\'{e}di's theorem, and discordant sets.
We observed above, using van der Waerden's theorem, that any piecewise syndetic set is AP-rich.
On the other hand, given that every piecewise syndetic set is AP-rich, the partition regularity of piecewise syndeticity tells us that, for any finite partition $\N=\bigcup_{i=1}^rC_i$, there is some $i$ for which $C_i$ is AP-rich. Thus van der Waerden's theorem is equivalent, in this informal sense, to the fact that every piecewise syndetic set is AP-rich.
% \color{green}In light of this, if one could prove that every discordant set is AP-rich, then Szemer\'edi's theorem would follow at once upon invoking van der Waerden's theorem.
Because of this, the existence of discordant sets can be interpreted as a sign that Szemer\'edi's infamously difficult theorem is deeper than van der Waerden's.\footnote
{Another equivalent form of Szemer\'edi's theorem says that any subset of $\N$ having positive upper Banach density is AP-rich.
In contrast to density, for which one averages along the intervals $\{1, \dots, n\}$, the \textit{upper Banach density} of a set $A$ is defined to be the supremum of all densities of $A$ averaged along any sequence of intervals of increasing length.
Any piecewise syndetic set has positive upper Banach density, a fact which clarifies the relationship between van der Waerden's and Szemer\'edi's theorems and strengthens the perspective we take here.}
The striking abundance of contexts in which discordant sets arise, some of which will be provided in the ensuing sections, can be taken as evidence for the versatility of Szemer\'edi's theorem and strengthens our interest in discordant sets.

The distinction between van der Waerden's theorem and Szemer\'edi's theorem can also be profitably analyzed from the point of view of dynamics and ergodic theory. Indeed, as initiated by Furstenberg?s ergodic-theoretic proof of Szemer\'edi?s theorem \cite{Fur2}, one can elucidate the properties of a subset of $\N$ by examining the dynamical and ergodic properties of the orbit closure of its indicator function.
To be more explicit, given $A\subseteq\N$, consider its indicator function $\mathbf1_A\in\{0,1\}^\N$, where we view $\{0,1\}^\N$ as having the product topology (making $\{0,1\}^\N$ homeomorphic to the Cantor set). Then, letting $T \colon \{0,1\}^\N\to\{0,1\}^\N$ denote the left-shift map (i.e., $T(a_1,a_2,\dots) = (a_2,a_3,\dots))$, define the \textit{orbit closure} of $\mathbf{1}_A$ as
\[
\mathcal{O}(\mathbf{1}_A) = \overline{\{T^n(\mathbf1_A)\colon n\in\N\}}.
\]
The resulting topological dynamical system ($\mathcal O(\mathbf1_A), T)$ can be made into a measure-preserving system.
In this setting, the difference in difficulty between van der Waerden's and Szemer\'edi's theorems is indicated by the existence or nonexistence of nontrivial minimal subsystems of $\mathcal O(\mathbf1_A)$.
In the presence of a nontrivial minimal subsystem, a purely topological-dynamical recurrence theorem, the topological analogue to van der Waerden's theorem, can be used to obtain a streamlined proof of the classical van der Waerden theorem \cite[Theorem 0.3]{FW}.
On the other hand, when $\mathcal{O}(\mathbf1_A)$ does not possess nontrivial minimal subsystems, a considerably more difficult proof via measurable dynamics is necessary to prove Szemer\'edi's theorem.
In particular, those sets $A \subseteq \N$ for which $\mathcal{O}(\mathbf1_A)$ has a nontrivial minimal subsystem are exactly the piecewise syndetic sets (Theorem \ref{pssubsystems}). This dynamical and ergodic approach leads not only to a proof of these classical theorems but also to far-reaching generalizations which still elude purely combinatorial proofs, such as versions of the polynomial Szemer\'edi theorem (e.g. \cite[Theorem 1.1]{BLL} and \cite[Theorem 7.3]{BM}).

Finally, because the existence of discordant sets is a nontrivial byproduct of the interplay between the notions of density and piecewise syndeticity, both very basic definitions of Ramsey theory, these special sets are particularly nice vehicles for advancing Ramsey theory in different settings and showing the bountiful connections between Ramsey theory and other areas of mathematics. Therefore our paper is not only about discordant sets but also is meant as an accessible (but admittedly very non-exhaustive) demonstration of the diversity of Ramsey theory.

\subsection*{Outline of the paper} This paper concerns itself with the study of subsets of  countably infinite  semigroups, so some definitions must be introduced to allow us to work in this generality. This is done in Section \ref{the-basics}, where we also provide the necessary prerequisites from Ramsey theory, ergodic theory, and topological dynamics.

In Section \ref{irrational-rotation} we give an ergodic-theoretic construction of discordant sets that can be done in many semigroups, generalizing a construction based on irrational rotation of the torus $\R/\Z$. In particular, we prove that the sets of return times to certain nowhere dense sets of positive measure form discordant sets. This section lays the foundation for the analogy between nowhere dense sets of positive measure and discordant sets, which we revisit and reexamine several times.

In Section \ref{number-theory}, we prove in two ways that the squarefree numbers are discordant.
The first proof is purely algebraic while the second uses the dynamical results of the previous section.
We also give several wide-ranging generalizations of this fact.

Our view narrows in Section \ref{curious-settings}, where we take a brief excursion to $\SL_2(\Z)$.
Although the classical notions of density fail to exist in this group since $\SL_2(\Z)$ is not amenable, we define a natural analogue of density which allows us to construct discordant sets using the results of the preceding section.

%Our study of discordant sets culminates in Section \ref{most-settings}, in which we put Section \ref{number-theory}'s characterization to work to prove that discordant sets can be found in any countable amenable cancellative semigroup.

In Section \ref{recurrence}, we study a special kind of discordant set which is particularly bad for recurrence in measure-preserving systems. We establish the existence of such discordant sets with density $c$ for each $c\in(0,1)$ for any  countably infinite  abelian group and also prove the existence of discordant sets with density $c$ for each $c\in (0,1)$ in any countably infinite commutative cancellative semigroup.

We shift focus in Section \ref{topology}. First, we generalize the notion of a \textit{disjunctive $\{0,1\}$-sequence}, one in which every finite $\{0,1\}$-word occurs, to an arbitrary  countably infinite  semigroup, and show that such sequences are ``topologically typical'' in the sense that the family of such sequences is comeager. This allows us to state that, in cancellative semigroups, non-piecewise syndetic sets (and hence discordant sets) are topologically atypical. We then consider special types of disjunctive sequences by measuring the limiting frequency with which every finite word occurs. In contrast to how one defines a \textit{normal sequence} by saying that these frequencies approach $2^{-n}$, where $n$ is the length of the word, we define an \textit{extremely non-averageable}, or \textit{ENA}, sequence by declaring that these frequencies should fail to converge as badly as possible. We prove that ENA sequences are topologically typical in arbitrary  countably infinite  amenable cancellative semigroups, strengthening the previous result for such semigroups.

Finally, in Section \ref{topol-dynamics} we give a topological-dynamical characterization of non-piecewise syndetic sets in arbitrary  countably infinite  semigroups by looking at their orbit closures. We conclude the paper by giving combinatorial applications of these dynamical results.

\subsection*{Acknowledgements}
We would like to thank Dona Strauss for helpful conversations. Her input was instrumental in guiding us toward the proof of Theorem \ref{BetterARall}.
We are also indebted to the anonymous referee for their numerous helpful comments and corrections.
The latter two authors were partially supported by the Ohio State Arts and Sciences Undergraduate Research Scholarship.

\section{Preliminaries}\label{the-basics}

Throughout this paper, $G$ will stand for a  countably infinite  semigroup. 

When we say that $G$ is a \textit{left cancellative} semigroup, we mean that for any $g,h_1,h_2\in G$, if $gh_1=gh_2$ then $h_1=h_2$; $G$ is \textit{right cancellative} if instead $h_1g=h_2g$ implies $h_1=h_2$; and $G$ is \textit{cancellative} if it is both left and right cancellative.

Given $g\in G$, define $gA=\{ga\colon a\in A\}$ and $g^{-1}A=\{x\in G\colon gx\in A\}$; the sets $Ag$ and $Ag^{-1}$ are defined analogously. Also, given $B\subseteq G$, define \[AB=\{ab\colon a\in A,b\in B\},\quad A^{-1}B=\bigcup_{a\in A}a^{-1}B,\quad\text{and}\quad AB^{-1}=\bigcup_{b\in B}Ab^{-1}.\] 

For any set $X$, we will write $\mathcal P(X)$ to denote the set of subsets of $X$. %and $\mathcal P_\mathrm{f}(X)$ to denote the set of finite subsets of $X$.
We write $Y^\mathrm{c}$ for the complement of $Y$, when the superset is understood. In particular, for $A\subseteq G$, $A^\mathrm{c}$ always denotes the complement relative to $G$.

\subsection{Thickness, syndeticity, and piecewise syndeticity} Let $A\subseteq G$.
\begin{itemize}
\item $A$ is \textit{right thick} if for any finite $F\subseteq G$ there exists $g\in G$ such that $Fg\subseteq A$.
\item $A$ is \textit{right syndetic} if there exists some finite $H\subseteq G$ such that $G=H^{-1}A$.
\item $A$ is \textit{right piecewise syndetic} if there is a finite $H\subseteq G$ for which $H^{-1}A$ is right thick.
\end{itemize}
Since $G$ is not assumed to be commutative, there are also ``left'' versions of these definitions obtained by replacing $Fg$ with $gF$ and $H^{-1}A$ with $AH^{-1}$.
% \begin{itemize}
% \item $A$ is \textit{left thick} if for any finite $F\subseteq G$ there exists $g\in G$ such that $gF\subseteq A$.
% \item $A$ is \textit{left syndetic} if there exists some finite $H\subseteq G$ such that $G=AH^{-1}$.
% \item $A$ is \textit{left piecewise syndetic} if there is a finite $H\subseteq G$ for which $AH^{-1}$ is left thick.
% \end{itemize}
We warn the reader that the left and right versions do not necessarily coincide in noncommutative (semi)groups.\footnote{Here are some examples to show that the left- and right-handed notions really are different.
In $F_2$, the free group on $a$ and $b$, it is easily seen that $F_2a$ is right thick but not left thick, and left syndetic but not right syndetic.
For a less noncommutative example, consider the following subset of the Heisenberg group $H_3(\Z)$:
\[A={\left\{{\left(\begin{smallmatrix}1&x&y\\0&1&z\\0&0&1\end{smallmatrix}\right)}\in H_3(\Z)\colon|y|\leq z\right\}}.\]
It can be shown that $A$ is left thick but not right piecewise syndetic (and thus not right thick). The group $H_3(\Z)$ is ``almost abelian'', having nilpotency class $2$, so we see from this example that not much noncommutativity is needed to separate the right- and left-handed notions.
See \cite[Section 2]{BHM} for many similar examples.}
Henceforth we will deal only with the right-handed definitions and omit the word ``right''.

Note that, if $(F_n)$ is an increasing sequence of finite subsets of $G$ satisfying $G=\bigcup F_n$, then $A\subseteq G$ is thick if and only if $\bigcup F_ng_n\subseteq A$ for some sequence $(g_n)$ in $G$.

We now outline the relationship between thick, syndetic, and piecewise syndetic sets. First, the family of syndetic sets and the family of thick sets are \textit{dual}, in the following sense.

\begin{lemma}\label{ST-duality}
Let $A\subseteq G$.
Then $A$ is syndetic if and only if $A^\mathrm{c}$ is not thick.
\end{lemma}

\begin{proof}
Suppose $A$ is syndetic, so $G = H^{-1}A$ for some finite $H \subseteq G$.
Let $g \in G$.
Since $g \in H^{-1}A$, there exists $h \in H$ with $hg \in A$.
Hence $Hg \cap A \neq \varnothing$, i.e., $Hg \nsubseteq A^\mathrm{c}$.
This proves that $A^\mathrm{c}$ is not thick.

Conversely, suppose $A$ is not syndetic, and let $F \subseteq G$ be finite. Since $F^{-1}A \neq G$, find $g\in G$ such that $g\notin F^{-1}A$.
Then $Fg\subseteq A^\mathrm{c}$, so $A^\mathrm{c}$ is thick.
\end{proof}

From this duality we obtain a useful equivalent description of piecewise syndetic sets, which we described in introduction for subsets of $(\N,+)$.

\begin{lemma}\label{intersections}
Let $A\subseteq G$.
Then $A$ is piecewise syndetic if and only if there exist a syndetic set $S\subseteq G$ and a thick set $T\subseteq G$ such that $A=S\cap T$.
\end{lemma}

\begin{proof}
First suppose $A$ is piecewise syndetic, and find a finite $H\subseteq G$ such that $H^{-1}A$ is thick. Let $T=A\cup H^{-1}A$, and let $S=A\cup T^\mathrm{c}$, so that $A=S\cap T$.
Pick any $x \in G$; we will show that $G = (Hx \cup \{x\})^{-1}S$.
Indeed, let $g \in G$.
If $g \notin x^{-1}S$, then $xg \in S^\mathrm{c} = A^\mathrm{c} \cap T \subseteq H^{-1}A$, so $g \in (Hx)^{-1}S$.
%We wish to show that $S$ is syndetic, so suppose this is not the case, meaning that $S^\mathrm{c}$ is thick. Fix $x\in G$; then for some $g\in G$ we have $(Hx\cup\{x\})g\subseteq S^\mathrm{c}$. Since $xg\in S^\mathrm{c}=A^\mathrm{c}\cap T\subseteq H^{-1}A$, for some $h\in H$ we have $hxg\in A$, a contradiction, since $hxg\in S^\mathrm{c}\subseteq A^\mathrm{c}$.

Now suppose $S\subseteq G$ is syndetic and $T\subseteq G$ is thick. Fix a finite $H\subseteq G$ satisfying $G=H^{-1}S$. Let $F\subseteq G$ be finite and pick $g\in G$ satisfying $HFg\subseteq T$. By the syndeticity of $S$, for any $f\in F$ we can find $h\in H$ such that $fg\in h^{-1}S$.
We also know that $hfg\in T$, so $fg \in h^{-1}(S \cap T)$. This proves that $Fg\subseteq H^{-1}(S\cap T)$, and hence that $H^{-1}(S\cap T)$ is thick.
\end{proof}

\subsection{Densities in semigroups}\label{densities}

In this subsection, $G$ is assumed to be a  countably infinite  amenable cancellative semigroup. Given these assumptions, amenability is equivalent to the existence of a F\o lner sequence.\footnote
{Amenability is commonly defined as follows. For each $g\in G$, define $\lambda_g\colon G\to G$ by $\lambda_g(h)=gh$. Then $G$ is called (\textit{left}) \textit{amenable} if there exists a positive normalized linear functional $L$ on the space $B(G)$ of bounded functions $G\to\R$ which satisfies $L(f\circ\lambda_g)=L(f)$ for each $f\in B(G)$; $L$ is called an \textit{invariant mean} for $G$. For  countably infinite  left cancellative semigroups, the existence of F\o lner sequences and amenability are equivalent \cite{AW}. See \cite{Pa} for a comprehensive study of the notion of amenability.}
A (\textit{left}) \textit{F\o lner sequence} in $G$ is a sequence $\Phi=(\Phi_n)$ of nonempty finite subsets of $G$ satisfying
\[\lim_{n\to\infty}\frac{|\Phi_n\cap g\Phi_n|}{|\Phi_n|}=1,
\quad\text{or equivalently}\quad
\lim_{n\to\infty}\frac{|\Phi_n\bigtriangleup g\Phi_n|}{|\Phi_n|}=0,\]
for any $g\in G$. The \textit{upper density} of $A\subseteq G$ with respect to $\Phi$ is defined as \[\overline d_\Phi(A)=\limsup_{n\to\infty}\frac{|A\cap \Phi_n|}{|\Phi_n|}.\] When the corresponding limit exists this quantity is the \textit{density} of $A$ with respect to $\Phi$, denoted by $d_\Phi(A)$, and we also define the \textit{lower density} $\underline d_\Phi(A)$ to be the corresponding $\liminf$. Note that $\underline d_\Phi(A)=1-\overline d_\Phi(A^{\mathrm c})$.
Observe also that it follows quickly from the pigeonhole principle and the cancellativity of $G$ that F\o lner sequences satisfy $|\Phi_n|\to\infty$, guaranteeing a degree of non-degeneracy for these definitions.

A subset $A\subseteq G$ will be called \textit{discordant} if it is not piecewise syndetic and there is some F\o lner sequence $\Phi$ in $G$ such that $\overline d_\Phi(G)>0$. We would like to also guarantee, when possible, the stronger condition $\underline d_\Phi(A)>0$, or better, $d_\Phi(A)>0$.

In the introduction, we defined density in $(\N,+)$ with respect to the F\o lner sequence $(\{1,\dots,n\})_{n\in\N}$. We will reserve $d$, leaving the F\o lner sequence unspecified, to refer to this density (and similarly with $\overline d$ and $\underline d$). A set $A\subseteq\N$ for which $d(A)$ exists is sometimes said to have \textit{natural density}.
Closely related are densities defined with respect to the F\o lner sequences $(\{-n,\dots,n\})$ in $\Z$ and $(\{-n,\dots,n\}^d)$ in $\Z^d$; we will also use $d$ to refer to these without danger of ambiguity.

% A useful variant of density is \textit{upper Banach density}, defined for any $A\subseteq G$ as \begin{equation}\label{banach}
% d^*(A)=\sup\{\overline d_\Phi(A)\colon\text{$\Phi$ is a F\o lner sequence in $G$}\}.
% \end{equation}
% Although we will usually work with upper and lower density, upper Banach density will be useful for some technical lemmas in Section 6.

Upper density has some useful properties, which we briefly describe here.
The most primitive is that it is invariant under shifts.
In other words, for any $g\in G$ and $A\subseteq G$, one has $\overline d_\Phi(gA)=\overline d_\Phi(g^{-1}A)=\overline d_\Phi(A)$.
We record the following similar result as we will use it later on.
(The proof of shift-invariance can be seen in the proof of Lemma \ref{shift-additivity}.)

\begin{lemma}\label{shift-additivity}
Suppose $g\in G$ and $A \subseteq G$ satisfy $A\cap gA=\varnothing$.
Then $\overline d_\Phi(A\cup gA)=2\,\overline d_\Phi(A)$.
\end{lemma}

\begin{proof}
It is clear that $\overline{d}_\Phi$ is subadditive.
For the opposite inequality, let $\varepsilon>0$ and find $n \in \N$ such that
\[
\frac{|A \cap \Phi_n|}{|\Phi_n|} > \overline{d}(A) - \frac\varepsilon3
\quad\text{and}\quad
\frac{|g\Phi_n \setminus \Phi_n|}{|\Phi_n|} < \frac\varepsilon3.
\]
Then
\[
\frac{|(A \cup gA) \cap \Phi_n|}{|\Phi_n|}
\geq
\frac{|A \cap \Phi_n|}{|\Phi_n|}
+
\frac{|gA \cap g\Phi_n|}{|\Phi_n|}
-
\frac{|g\Phi_n \setminus \Phi_n|}{|\Phi_n|}.
\]
Now since $|gA \cap g\Phi_n| = |g(A \cap \Phi_n)| = |A \cap \Phi_n|$ by cancellativity, we get
\[
\frac{|(A \cup gA) \cap \Phi_n|}{|\Phi_n|}
>
{\left(\overline{d}_\Phi(A) - \frac{\varepsilon}{3}\right)}
+
{\left(\overline{d}_\Phi(A) - \frac{\varepsilon}{3}\right)}
-
\frac{\varepsilon}{3}
>
2\,\overline{d}_\Phi(A) - \varepsilon.
\qedhere
\]
\end{proof}

Finally, let us remark that density is related to thickness and syndeticity in the following way: $A$ is thick if and only if $\overline{d}_\Phi(A)=1$ for some $\Phi$, and $A$ is syndetic if and only if $\overline{d}_\Phi(A)>0$ for all $\Phi$; see \cite[Theorem 2.6]{BHM}.
More discussion of thickness, syndeticity, piecewise syndeticity, the various notions of density, and other definitions of Ramsey theory can be found in, among other places, \cite[Sections 1 and 2]{BHM}.

\subsection{Partition regularity and Straus' construction}\label{partition-regularity}

To begin this subsection we will prove two lemmas about piecewise syndeticity and partition regularity. The following lemma states that the family of piecewise syndetic subsets of any countably infinite semigroup is partition regular.
A different proof of this fact appears in \cite[Lemma 7.2]{Hi3}.
Since, as we will see, non-piecewise syndetic sets can be seen as discrete analogues of nowhere dense sets, Theorem \ref{brown} is a sort of discrete analogue to the Baire category theorem on unions of nowhere dense sets. We will not use Theorem \ref{brown} later on, but include it for completeness and to support the discussion given in the introduction and the analogy we are hoping to emphasize.
A different proof, using topological dynamics, will be presented in the final section (Theorem \ref{brown2}).

\begin{theorem}\label{brown}
Suppose $A\subseteq G$ is piecewise syndetic. Then for any finite partition $A=\bigcup_{i=1}^rC_i$ into disjoint sets, there exists $i$ for which $C_i$ is piecewise syndetic.
\end{theorem}

\begin{proof}
We will prove this for $r=2$, and the claim will follow by induction. By Lemma~\ref{intersections}, we may write $A=S\cap T$ for a syndetic set $S\subseteq G$ and thick $T\subseteq G$. Let $S'=C_2^\mathrm{c}\cap S$, so we get $C_1=S'\cap T$ (since $C_1$ and $C_2$ are disjoint) and $C_2=S\cap(S')^\mathrm{c}$. Since either $S'$ is syndetic or $(S')^\mathrm{c}$ is thick by Lemma \ref{ST-duality}, one of $C_1$ and $C_2$ must be piecewise syndetic.\footnote{Note that this proof does not use the definition of piecewise syndeticity, but only that the family of piecewise syndetic sets is the family of intersections of two dual families (``dual'' in the sense of Lemma \ref{ST-duality}). In other words, we have shown that if $\mathcal A,\mathcal B\subseteq\mathcal P(G)$ are any two dual families of subsets of $G$ that are upward closed, then the family $\{A\cap B\colon\text{$A\in\mathcal A$ and $B\in\mathcal B$}\}$ is guaranteed to be partition regular.}
\end{proof}

The following easy but important lemma demonstrates another connection between piecewise syndeticity and partition regularity. We record it as the general phenomenon behind why Straus' construction produces discordant sets. We will apply this lemma by starting with a family $\mathcal B\subseteq\mathcal P(G)$ and extending it to a $G$-invariant family $\mathcal A=\{gB\colon B\in\text{$\mathcal B$ and $g\in G$}\}$.

\begin{lemma}\label{PS-friendly}
Suppose $\mathcal A\subseteq\mathcal P(G)$ satisfies the following: 
\begin{enumerate}
\item $\mathcal A$ contains all thick subsets of $G$.
\item $\mathcal A$ is partition regular.
\item if $B\subseteq A\subseteq G$ and $B\in\mathcal A$, then $A\in\mathcal A$.
\item if $A\in\mathcal A$ and $g\in G$ then $gA\in\mathcal A$.
\end{enumerate}
Then $\mathcal A$ contains all piecewise syndetic subsets of $G$.
\end{lemma}

\begin{proof}
Let $A\subseteq G$ be piecewise syndetic and find a finite set $H\subseteq G$ such that $H^{-1}A$ is thick. Then $H^{-1}A=\bigcup_{h\in H}h^{-1}A$, so by (\textit{a}) and (\textit{b}), $h^{-1}A\in\mathcal A$ for some $h\in H$. Thus $h(h^{-1}A)\in\mathcal A$, so finally $A\in\mathcal A$ since $h(h^{-1}A)\subseteq A$.
\end{proof}

With this lemma, we give a simplified version of Straus' construction to obtain a set with positive lower density, no shift of which is IP.

\begin{theorem}\label{straus-construction}
Let $(a_n)$ be a sequence in $\N$ satisfying $\sum1/a_n<1$. Then the following set is discordant: \[A=\N\setminus\bigcup_{n\in\N}(a_n\N+n-1).\]
\end{theorem}

\begin{proof}
We leave it as an exercise to verify that every thick subset of $\N$ is IP; thus, by Hindman's theorem and Lemma~\ref{PS-friendly}, it suffices to show that $A$ is not a shift of an IP set (that is, a set of the form $B+t$ for some IP set $B\subseteq\N$ and $t\in\N$). This follows immediately from the construction of $A$ and the observation that if $B\subseteq\N$ is IP, then $B\cap k\N\neq\varnothing$ for all $k\in\N$.

To see that $\underline d(A)>0$, we compute that \[\frac{|A\cap\{1,\dots,k\}|}{k}\geq\frac{k-(k/a_1+k/a_2+\cdots+k/a_n)}{k}>1-\sum_{n\in\N}\frac{1}{a_n}\] for each $k<a_{n+1}$, since there are no more than $k/a$ multiples of $a$ in $\{1,\dots,k\}$.
\end{proof}

For conciseness, we excuse ourselves from proving that $d(A)$ exists, where $A$ is the set defined in Theorem \ref{straus-construction}. The reader is encouraged to look toward the sources we have mentioned or, alternatively, Theorem \ref{bfree-density} to see how this could be done.
One may also alter the set described in Theorem \ref{straus-construction} to obtain a discordant set without invoking Hindman's theorem as follows.
Pick $(a_n)$ satisfying $\sum n/a_n<1$ and set $A = \N \setminus \bigcup(a_n\N + \{0, \dots, n - 1\})$.
One shows similarly that $\underline d(A)>0$, and a totally self-contained proof that $A$ is not piecewise syndetic is not hard.

Note that Theorem \ref{straus-construction} gives us discordant sets with lower densities arbitrarily close to $1$. (A discordant set cannot have upper density equal to $1$, since any set of upper density $1$ is thick.) By generalizing the notion of an IP set, a similar construction can be done in many different groups. Discussion of this construction and a complete characterization of such groups is given in \cite[Chapter 4]{Ch} and \cite{BF}.

\subsection{Preliminaries from ergodic theory and dynamics}\label{ergodic-business}

In this section we review some basic notions and theorems of ergodic theory and dynamics. This section is utilitarian, and hence we do not venture past what we will need. For this reason we omit proofs or provide a rudimentary outline and refer the reader elsewhere.

Let $(X,\mathcal B,\mu)$ be a probability space. A measurable map $T\colon X\to X$ is \textit{measure preserving} if $\mu(T^{-1}E)=\mu(E)$ for each $E\in\mathcal B$. If $G$ is a  countably infinite  semigroup, a \textit{measure-preserving action} of $G$ on $(X,\mathcal B,\mu)$ is an action of $G$ on $X$ such that $x\mapsto gx$ is measure-preserving for each $g\in G$. In this case, we say that $\mu$ is \textit{$G$-invariant}. We refer to $(X,\mathcal B,\mu,T)$ and $(X,\mathcal B,\mu,G)$ as \textit{measure-preserving systems}.

Given a measure-preserving system, one has the following version of the von Neumann ergodic theorem.

\begin{theorem}\label{meanergodic} Suppose $G$ is a  countably infinite  amenable cancellative semigroup with a F\o lner sequence $(\Phi_n)$ and let $(X,\mathcal{B},\mu,G)$ be a measure-preserving system. Let $f\in L^2(X)$, and let $f^*$ be the orthogonal projection of $f$ onto the subspace of $L^2(X)$ consisting of $G$-invariant functions. Then \[\lim_{n\to\infty}\frac{1}{|\Phi_n|}\sum_{g\in\Phi_n}f(gx)=f^*(x)\quad\text{in $L^2(X)$.}\] 
\end{theorem}

\begin{proof}
See \cite[Theorem 1.14.1]{Wa} for a proof in the case $G=\mathbb{Z}$. The proof for more general $G$ is analogous; see \cite[Theorem 4.15]{Be}.
\end{proof}

With additional constraints on the action of $G$, more can be said.
The system $(X,\mathcal{B},\mu,G)$ is called \textit{ergodic} if there are no nontrivial $G$-invariant subsets, that is, if $E\in\mathcal{B}$ satisfies $\mu(E\bigtriangleup g^{-1}E)=0$ for all $g\in G$, then $\mu(E)\in\{0,1\}$.
In this case, the orthogonal projection of $f$ onto the subspace of $G$-invariant functions is $\int_Xf\,d\mu$, so Theorem \ref{meanergodic} states that \[\lim_{n\to\infty}\frac{1}{|\Phi_n|}\sum_{g\in\Phi_n}f(gx) = \int_X f\,d\mu \quad\text{in $L^2(X)$.}\]

If $X$ is a compact topological space on which $G$ acts by continuous maps, the pair $(X,G)$ is a (\textit{topological}) \textit{dynamical system}. If $E \subseteq X$ and $x \in X$, we define $R_E(x) = \{g \in G \colon gx \in E\}$.
When $G$ is $\Z$ or $\R$, $R_E(x)$ can be thought of as the set of visiting times of $x$ to $E$.
The system $(X, G)$ is \textit{minimal} if there is no proper nonempty closed subset $C\subseteq X$ invariant under the action of $G$ (in the sense that $GC\subseteq C$). Equivalently, the action of $G$ on $X$ is minimal if, for each $x \in G$, its orbit $\{gx \colon g \in G\}$ is dense.

%not metric space: convenient thing to use
When $X$ is a compact metric space with a Borel $\sigma$-algebra $\mathcal B$, we say that $(X, G)$ is \textit{uniquely ergodic} if there is exactly one $G$-invariant Borel probability measure on $X$, i.e., there exists exactly one normalized measure $\mu$ on $\mathcal B$ which makes $(X,\mathcal B,\mu,G)$ into a measure-preserving system.
Under the condition of unique ergodicity, one has the following convenient pointwise ergodic theorem.

\begin{theorem}\label{uniqueergodic}
Let $(X, G)$ be a uniquely ergodic topological dynamical system where $X$ is a compact metric space and $G$ is a  countably infinite  amenable cancellative semigroup.
Assume $\mu$ is the unique $G$-invariant Borel probability measure and let $(\Phi_n)$ be a F\o lner sequence in $G$.
Then, for each measurable $f\colon X\to\R$ whose points of discontinuity are contained in a set of measure zero and each $x\in X$, \[\lim_{n\to\infty}\frac{1}{|\Phi_n|}\sum_{g\in\Phi_n}f(gx)=\int f\,d\mu.\]
\end{theorem}

\begin{proof}
For continuous $f$, this can be proved in complete analogy with the well-known proof of the case $G=\Z$; see \cite[Theorem 6.19]{Wa}. For more general $f$, the statement follows by a standard approximation argument.
\end{proof}

A fact we will find useful is that a minimal action by isometries on a compact metric space is automatically uniquely ergodic.

\subsection{Cantor spaces}\label{cantor}

The following construction will be very useful.
Given any sequence $(X_n)_{n \in \N}$ of finite sets having the discrete topology, their topological product $X = \prod X_n$ is a compact space homeomorphic to the classical Cantor set.
We will use this construction frequently, and we will call such a space a \textit{Cantor space}. To make the notation less cluttered and easier to read, we will use boldface letters for elements of a Cantor space, e.g., $\mathbf{x} = (x_n) \in X$. (This is also consistent with our use of $\mathbf1$ for the indicator function.)

Given $n \in \N$ and $y_i \in X_i$ for $i \in \{1, \dots, n\}$, we define the set
\[V(y_1,\dots,y_n) = \{\mathbf{x} \in X\colon x_i=y_i\text{ for each }i \in \{1,\dots,n\}\}.\]
The collection of all sets having this form constitute an open base for the topology on $X$.

It will be convenient to also give Cantor spaces a metric space structure.
Giving each $X_n$ the discrete metric $\delta_n$, the space $X$ can be made into a compact metric space via the product metric $\delta$ given by
\begin{equation}\label{product-metric}
\delta(\mathbf{x},\mathbf{y})
= \sum_{n \in \N} 2^{-n}\delta_n(x_n,y_n).
\end{equation}
Assuming $\mathbf{x} \neq \mathbf{y}$, let $n \in \N$ be the smallest integer satisfying $x_n = y_n$ and $x_{n + 1} \neq y_{n + 1}$ (take $n = 0$ if $x_1 \neq y_1$). Then $\delta$ is equivalent to the more common and conceptually simpler metric $\delta'$ defined by $\delta'(\mathbf{x}, \mathbf{y}) = 1/(n + 1)$.
Observe that, given a collection of bijections $f_n \colon X_n \to X_n$ for each $n \in \N$, the induced map $f \colon X \to X$ is an isometry of $X$.

There is one more construction that we describe here.
Analogously to the natural topological and metric structures on $X$, we can endow each $X_n$ with the normalized counting measure.
The resulting (uniform) product measure $\mu$ is a probability measure on the Borel $\sigma$-algebra $\mathcal{B}$ of $X$ satisfying $\mu(V(y_1,\dots,y_n)) = 1/(|X_1| \cdots |X_n|)$.
Clearly, $\mu$ is invariant under the isometries $f$ described in the previous paragraph.

\section{Discordant sets from visiting times}\label{irrational-rotation}

As foreshadowed by the introduction, discordant sets naturally arise from dynamical systems, one of the most familiar being irrational rotation of the one-dimensional torus $\T=\R/\Z$.
Our search for discordant sets is motivated by the fact that if $\alpha\in\T$ is irrational and $U\subseteq\T$ has nonempty interior, then $\{n\in\Z\colon n\alpha\in U\}$ is syndetic.
The punchline of this direction of reasoning is that if $E\subseteq\T$ is nowhere dense and has positive measure, then for almost every $x\in\T$, the set of visiting times $R_E(x)=\{n\in\Z\colon x+n\alpha\in E\}$ is discordant.
Here $R_E(x)$ is non-piecewise syndetic for any $x\in\T$ by Lemma \ref{not-nowhere-dense} below, whereas the positive upper density of $R_E(x)$ is guaranteed only for almost every $x$ by an application of the von Neumann ergodic theorem.\footnote
{In fact, one can easily see that $R_E(x)$ has positive \textit{natural} density for almost every $x$ by using the pointwise ergodic theorem.
}
Theorem~\ref{dynamical-destruction} is a generalization of this construction which yields many discordant sets in many settings, and is the core of our analogy between discordant sets and nowhere dense sets of positive measure.

Let us note that the restriction ``almost every $x$'' is a significant and unavoidable constraint. One can construct a closed nowhere dense set $E\subseteq\T$ of positive measure for which $\{n\in\Z\colon x+n\alpha\in E\}$ is empty for infinitely many $x\in\T$.

The importance of the irrationality of $\alpha$ is that $(\T, x \mapsto x + \alpha)$ is minimal if and only if $\alpha$ is irrational. In our more general setting we replace this system with a minimal action of a semigroup $G$ on a compact space $X$.

In this setting, we wish to connect properties of $R_E(x) = \{g \in G \colon gx \in E\}$ as a subset of $G$ with the properties of $E$ as a subset of $X$. Lemma \ref{not-nowhere-dense} connects the piecewise syndeticity of $R_E(x)$ with the topological properties of $E$ and Theorem \ref{dynamical-destruction} connects the density of $R_E(x)$ with the measure of $E$, allowing us to construct discordant sets.

In analogy with our notation for subsets of $G$, we write $A^{-1}E = \bigcup_{a \in A} a^{-1}E$ for $A \subseteq G$ and $E \subseteq X$.

\begin{lemma}\label{not-nowhere-dense}
Let $(X,G)$ be a minimal topological dynamical system, where $G$ is a  countably infinite  cancellative semigroup. Then, for any nowhere dense $E\subseteq X$ and $x\in X$, the set $R_E(x)$ is not piecewise syndetic.
\end{lemma}

\begin{proof}
It is a classical fact that $(X,G)$ is minimal if and only if $R_V(x)$ is syndetic for each nonempty open $V\subseteq X$ and $x\in X$; see, e.g., \cite[Theorem 1.15]{Fur}. We reproduce the argument here for the convenience of the reader. The ``if'' direction is evident, so suppose $V\subseteq X$ is open and nonempty and $x\in X$.
Let $Y=G^{-1}V$; we claim that $Y=X$. Indeed, if $y\in X\setminus Y$, then $gy\notin V$ for all $g$, so $y$ cannot have dense orbit, contradicting the minimality of $(X,G)$. By compactness, there is a finite set $H\subseteq G$ such that $X=H^{-1}V$.
Now fix $g\in G$ and find $h\in H$ such that $gx\in h^{-1}V$. Then $hg\in R_V(x)$, so $G=H^{-1}R_V(x)$, proving that $R_V(x)$ is syndetic.

Now let $E\subseteq X$ be nowhere dense and let $F\subseteq G$ be finite. Then $F^{-1}R_E(x)= R_{F^{-1}E}(x)$, and $F^{-1}E$ is nowhere dense. Let $V$ be an open set lying in $X\setminus F^{-1}E$. Then $F^{-1}R_E(x)$ is disjoint from the syndetic set $R_V(x)$, and so cannot be thick. Since $F$ was arbitrary, $R_E(x)$ cannot be piecewise syndetic.
\end{proof}

Suppose $(X, G)$ is a topological dynamical system.
If $G$ is amenable, there always exists a $G$-invariant measure $\mu$ on $X$ by a version of the Bogolyubov--Krylov theorem (see \cite[Theorem 6.9]{Wa} for a proof of the case $G=\Z$; the theorem for general $G$ is proved analogously).
Moreover, $\mu$ can taken to be ergodic by invoking the so-called ergodic decomposition.

\begin{theorem}\label{dynamical-destruction}
Let $(X,G)$ be a minimal topological dynamical system, where $G$ is a  countably infinite  amenable cancellative semigroup with a F\o lner sequence $\Phi = (\Phi_n)$.
Let $\mu$ be a $G$-invariant probability measure on the Borel $\sigma$-algebra $\mathcal{B}$ of $X$. If $E\subseteq X$ is nowhere dense with positive measure, then there exists $z\in X$ such that $R_E(z)$ is discordant.
If, in addition, $(X, \mathcal{B}, \mu, G)$ is ergodic, $\overline{d}_\Phi(R_E(z)) \geq \mu(E)$ for $\mu$-almost every $z \in X$.
\end{theorem}

\begin{proof}
Let $f = \mathbf{1}_E$. By Theorem \ref{meanergodic}, \[\lim_{n\to\infty}\frac{1}{|\Phi_n|}\sum_{g\in \Phi_n}f(gx)=f^*(x),\] where $f^*$ is the orthogonal projection of $f$ onto the space of $G$-invariant functions in $L^2(X,\mu)$, and the convergence is in $L^2$. Extract a subsequence $\Phi'=(\Phi_{n_k})_{k\in\mathbb{N}}$ such that  $\lim_{k\to\infty}|\Phi_{n_k}|^{-1}\sum_{g\in \Phi_{n_k}}f(gx)=f^*(x)$ almost everywhere.
For each $k$,
\[\int \frac{1}{|\Phi_{n_k}|}\sum_{g\in\Phi_{n_k}}f(gx)\,d\mu=\frac{1}{|\Phi_{n_k}|}\sum_{g\in\Phi_{n_k}}\int f(gx)\,d\mu=\frac{1}{|\Phi_{n_k}|}\sum_{g\in\Phi_{n_k}}\mu(E)=\mu(E).\]
Hence, by the dominated convergence theorem, $\int f^*\,d\mu=\mu(E)$. Thus, there is a set of positive measure on which $f^*>0$. Choose $z$ such that $f^*(z)>0$ and \[d_{\Phi'}(R_E(z))=\lim_{k\to\infty}\frac{1}{|\Phi_{n_k}|}\sum_{g\in \Phi_{n_k}}f(gz)=f^*(z).\] Then $\overline{d}_{\Phi}(R_E(z))\geq d_{\Phi'}(R_E(z))=f^*(z)>0$, and, since $(X,G)$ is minimal, $R_E(z)$ is not piecewise syndetic by Lemma \ref{not-nowhere-dense}.

In the presence of ergodicity, $f^*$ is the constant function $\int f\,d\mu = \mu(E)$, so we may take $z$ to be any element for which $\lim_{k\to\infty}|\Phi_{n_k}|^{-1}\sum_{g\in \Phi_{n_k}}f(gz)=f^*(z)$.
As we noted, such $z$ form a subset of $X$ of full measure.
\end{proof}

It is worth noting that, in the setting of Theorem \ref{dynamical-destruction}, $X$ possesses a large supply of nowhere dense sets of positive measure if it is an infinite metric space.

\begin{theorem}
\label{thm:lots-of-fat-cantor}
Let $(X,G)$ and $\mu$ be as in Theorem \ref{dynamical-destruction}, and assume that $X$ is a metric space which is not finite.\footnote
{
It is evident from the proof of Theorem \ref{thm:lots-of-fat-cantor} that the condition ``metric space'' can be weakened slightly, at the cost of more cumbersome hypotheses.
}
Then for any $c\in(0,1)$, there exists a nowhere dense set $E \subseteq X$ with $\mu(E)=c$.
\end{theorem}

\begin{proof}
We first claim that $\mu$ is a non-atomic measure.
Let $x\in X$.
Then $\mu(\{gx\})=\mu(g^{-1}\{gx\})\geq\mu(\{x\})$ for any $g\in G$.
Since $(X,G)$ is minimal, $\{gx\colon g\in G\}$ is dense in $X$, and is in particular infinite.
It follows that $\mu(\{x\})=0$ since $\mu$ is a probability measure.

Now let $\{y_n\}_{n \in \N} \subseteq X$ be a countable dense subset of $X$.
Since $\mu$ is non-atomic, $\mu(\{y_n\}) = 0$ for all $n \in \N$.
Hence by the continuity of $\mu$ from above, we can find an open neighborhood $U_n \subseteq X$ of $y_n$ with $\mu(U_n) < (1-c)/2^n$.
Set $E_0 = X \setminus \bigcup U_n$, so that $E_0$ is a nowhere dense subset of $X$ with $\mu(E_0) > c$.
The claim now follows from a theorem of Sierpi\'nski \cite{Si} which states that in any non-atomic measure space $(Y,\mathcal{N},\nu)$ with $\nu(Y)<\infty$, the measure $\nu$ takes every value in the interval $[0,\nu(Y)]$.
\end{proof}

\section{Discordant sets and generalizations of squarefree numbers}\label{number-theory}

It is known that the set of squarefree numbers in $\N$ is discordant, and in this section we describe a wide-ranging generalization of this fact and examine in detail some special cases.

\subsection{A characterization and a construction}

We will prove non-piecewise syndeticity in two ways; the first uses a characterization (Lemma \ref{characterization}) of non-piecewise syndetic sets, and the second uses Lemma \ref{not-nowhere-dense}. Ramsey theory aficionados will notice that Lemma \ref{characterization} can be rephrased as a characterization of sets satisfying the dual property, in the sense of Lemma \ref{ST-duality}, to piecewise syndeticity, often denoted by PS$^*$.

\begin{lemma}\label{characterization}
Let $G$ be a  countably infinite  semigroup and let $A \subseteq G$.
Then $A$ is not piecewise syndetic if and only if, for any finite $H \subseteq G$, there exists a syndetic $S \subseteq G$ such that $HS \cap A = \varnothing$.
\end{lemma}

\begin{proof}
Let $H \subseteq G$ be finite and consider the set $B_H = \{g \in G \colon Hg \cap A = \varnothing\}$.
Observe that $B_H^\mathrm{c}=H^{-1}A$, since $g \in B_H^\mathrm{c}$ if and only if $hg \in A$ for some $h \in H$.
Thus, by Lemma \ref{ST-duality}, $B_H$ is syndetic if and only if $H^{-1}A$ is not thick.
In particular, $B_H$ is syndetic for all finite $H \subseteq G$ if and only if $A$ is not piecewise syndetic.
\end{proof}

Before constructing non-piecewise syndetic sets using Lemma \ref{characterization}, we first recall a general group-theoretic form of the Chinese remainder theorem.\footnote{A generalization of this form of the Chinese remainder theorem is described in the StackExchange discussion \href{https://math.stackexchange.com/questions/2954401/validation-for-a-conjecture-about-chinese-remainder-theorem-for-groups}{\texttt{SE2954401}}, although we will not need this generality.}

\begin{lemma}[Chinese remainder theorem]\label{chinese}
Let $G$ be a group with normal subgroups $N_1,\dots,N_n$, where $n\geq2$ and $|G/N_1|,\dots,|G/N_n|$ are pairwise coprime.
Define $N=\bigcap_{i=1}^nN_i$ and $\varphi\colon G/N\to\prod_{i=1}^nG/N_i$ by $\varphi(gN)=(gN_1,\dots,gN_n)$.
Then $\phi$ is an isomorphism.
\end{lemma}

\begin{proof}
It is immediate from the first isomorphism theorem that $\phi$ is a well-defined injective homomorphism.
So, it suffices to show that $|G/N|\geq\prod_{i=1}^n |G/N_i|$.
Indeed, for each $i$, $|G/N_i|$ divides $|G/N|$ since $|G/N|=|G/N_i|\cdot|N_i/N|$.
Thus the inequality holds since $|G/N_1|,\dots,|G/N_n|$ are pairwise coprime.
\end{proof}

We now combine the Chinese remainder theorem with Lemmas \ref{characterization} and \ref{not-nowhere-dense} to obtain an abundance of non-piecewise syndetic sets. Although Lemma \ref{characterization} shows that we can remove ``thicker and thicker'' syndetic sets to produce non-piecewise syndetic sets, Theorem \ref{bfreegeneral} shows that in special situations (when the syndetic subsets are suitably ``arithmetically/algebraically independent'', as the Chinese remainder theorem will guarantee) this is not necessary.

\begin{theorem}\label{bfreegeneral}
Let $G$ be a  countably infinite  group and $(N_n)$ be a sequence of normal finite-index subgroups of $G$ such that for any $n\geq2$, the subgroups $N_1,\dots,N_n$ satisfy the hypotheses of Lemma \ref{chinese}. Then, for any sequence $(g_n)$ in $G$, the set $A=G\setminus\bigcup g_nN_n$ is not piecewise syndetic.
\end{theorem}

In principle, the set $A$ may be empty. Below we will see many examples where this is not the case. Notice that we can equivalently state the conclusion of Theorem \ref{bfreegeneral} as ``every piecewise syndetic subset of $G$ contains an element (in fact, infinitely many elements) of $\bigcup g_nN_n$''. We will give two proofs of this fact.

\begin{proof}[Proof 1]
Fix $F=\{f_1,\dots,f_n\} \subseteq G$.
Let $N=\bigcap_{i=1}^nN_i$, and let $\varphi$ be the isomorphism guaranteed by the Chinese remainder theorem. Then, setting $x=\varphi^{-1}(f_1^{-1}g_1N_1,\dots,f_n^{-1}g_nN_n)$, we have $f_ix\in g_iN_i$ for each $i\in\{1,\dots,n\}$, so $F(xN)\cap A=\varnothing$. Since the subgroups $N_i$ have finite index, $N$ is syndetic, so Lemma \ref{characterization} completes the proof.
\end{proof}

Our second proof of Theorem \ref{bfreegeneral}, though longer, uses Lemma \ref{not-nowhere-dense} and strengthens our analogy between discordant sets and nowhere dense sets of positive measure.
In the remainder of this subsection we will use the definitions made in Subsections \ref{ergodic-business} and \ref{cantor}.

\begin{proof}[Proof 2]
Consider the Cantor space $X=\prod G/N_n$.
Given $h_1,\dots,h_n\in G$, define
\begin{equation}\label{xbase}
V(h_1,\dots,h_n)=\{\mathbf{x}\in X\colon\text{$x_iN_i=h_iN_i$ for each $i\in\{1,\dots,n\}$}\}.
\end{equation}
Sets having the form given in \eqref{xbase} constitute an open base for the topology on $X$.

Let $G$ act by componentwise multiplication on $X$, i.e., $g\mathbf{x}=(gx_nN_n)_{n\in\N}$. The Chinese remainder theorem guarantees that this action is minimal. To see this, let $\mathbf{x}\in X$ and $h_1,\dots,h_n\in G$ be given. By the Chinese remainder theorem, find $h\in G$ so that $(hN_1,\dots,hN_n)=(h_1x_1^{-1}N_1,\dots,h_nx_n^{-1}N_n)$. Then $h\mathbf{x}\in V(h_1,\dots,h_n)$ since $hx_iN_i=hN_ix_i=h_ix_i^{-1}N_ix_i=h_iN_i$. This proves that $\mathbf{x}$ has dense orbit.

To apply Lemma \ref{not-nowhere-dense}, we simply observe that $A=R_E((g_n^{-1}N_n)_{n\in\N})$, where
\begin{equation}
\label{e}
E=\{\mathbf{x}\in X\colon x_nN_n\neq N_n\text{ for each }n\in\N\}.
\end{equation}
The set $E$ is nowhere dense, since it is closed and is not a superset of any set given in \eqref{xbase}.
\end{proof}

As one might expect from this second proof, we can obtain a density result for Theorem \ref{bfreegeneral} by applying Theorem \ref{dynamical-destruction}.

\begin{theorem}\label{bfreegeneraldensity}
Let $G$, $X$, and $(N_n)$ be as in Proof 2 of Theorem \ref{bfreegeneral}, and set $c_n=|G/N_n|$ for each $n\in\N$.
Suppose $G$ is amenable, possessing a F\o lner sequence $\Phi$.
Let $\mu$ be the probability measure on the Borel $\sigma$-algebra $\mathcal B$ of $X$ satisfying $\mu(V(h_1,\dots,h_n))=1/(c_1\cdots c_n)$.
Finally, given $\mathbf{x} \in X$, set $A_\mathbf{x} = G \setminus \bigcup x_nN_n$.
Then, for $\mu$-almost every $\mathbf{x} \in X$, \begin{equation}\label{actuallybfgd}
\overline{d}_\Phi(A_\mathbf{x})=\prod_{n\in\N}{\left(1-\frac{1}{c_n}\right)}.
\end{equation}
\end{theorem}

\begin{proof}
Since $G$ acts minimally on $X$ by isometries, the system $(X,\mathcal B,\mu,G)$ is uniquely ergodic, and in particular ergodic.
Let $E \subseteq X$ be as in \eqref{e}, so that $A_\mathbf{x}=R_E((x_n^{-1}N_n)_{n\in\N})$. Then by Theorem \ref{dynamical-destruction}, we have, for $\mu$-almost every $\mathbf{x}\in X$,
\[\overline{d}_\Phi(A_\mathbf{x})=\mu(E)=\prod_{n\in\N}{\left(1-\frac{1}{c_n}\right)}.\]
(That $\mu(E)=\prod(1-1/c_n)$ can be easily checked since $E$ is closed.)
\end{proof}

In practice, Theorem \ref{bfreegeneraldensity} will not be useful for constructing explicit discordant sets as we cannot easily control for which $\mathbf{x}$ equation \eqref{actuallybfgd} holds.

\subsection{Some number-theoretic examples of discordant sets}\label{nt-examples}

In this subsection we apply Theorem \ref{bfreegeneral} to construct some discordant sets.
In the following subsection we provide proofs of positive density for all but one of the constructions.
The most primitive example is the following.

\begin{example}\label{sqfree}
The set $Q$ of squarefree integers is not piecewise syndetic: \[Q=\Z\setminus\bigcup_{\text{$p$ prime}}p^2\Z.\] It is classical that $d(Q)=6/\pi^2$, so that $Q$ is discordant.
\end{example}

Example \ref{sqfree} can be generalized a great deal. We provide five generalizations below.

\begin{example}\label{bfree}
Let $\mathscr{B}=(b_n)$ be a sequence of pairwise coprime positive integers and let $\mathcal{F}_\mathscr{B}$ denote the set of \textit{$\mathscr{B}$-free} integers \[\mathcal{F}_\mathscr{B}=\Z\setminus\bigcup_{n\in\N}b_n\Z,\] consisting of integers not divisible by any element of $\mathscr{B}$. Then $\mathcal{F}_\mathscr{B}$ is not piecewise syndetic, and if $\sum1/b_n<\infty$, then $d(\mathcal{F}_\mathscr{B})=\prod(1-1/b_n)$ \cite[Theorem 9, Section V.5]{HR}. We provide a proof of this fact below (Theorem \ref{bfree-density}). A well-known but notable special case is that the set of \textit{$k$-free integers}, those divisible by no perfect $k$th power, has density $1/\zeta(k)$.
\end{example}

\begin{example}\label{kfree}
Let $K$ be a finite extension of $\Q$, let $k\geq2$, and let $Q_{K,k}$ denote the \textit{$k$-free} integers $K$: \[
Q_{K,k}=\mathcal O_K\setminus\bigcup_{\substack{\mathfrak p \subseteq \mathcal{O}_K \\ \text{prime ideal}}}\mathfrak p^k,\]
where $\mathcal{O}_K$ denotes the ring of integers in $K$.
Using an appropriate ring-theoretic version of the Chinese remainder theorem, we may apply Theorem \ref{bfreegeneral} to show that $Q_{K,k}$ is not piecewise syndetic. Moreover, as $K/\Q$ is a finite extension, $\mathcal O_K\cong\Z^d$ as additive groups, where $d$ is the degree of $K/\Q$. Define $\Phi$ to be a sequence of balls (pulled back from $\Z^d$ to $\mathcal O_{K}$ under this isomorphism), with respect to the $L^1$ norm, centered at $0$ and having increasing radius. Then $d_\Phi(Q_{K,k})=1/\zeta_K(k)$ as shown in \cite[Corollary 4.6]{CV}, where $\zeta_K$ denotes the Dedekind zeta function of $K$. Equivalently, $d_\Phi(Q_{K,k})=\prod_{\mathfrak p}(1-1/N(\mathfrak p)^k)$, where $N(\mathfrak p)$ denotes the \textit{absolute norm} of $\mathfrak p$, i.e., the cardinality of the residue field $\mathcal O_K/\mathfrak p$. Hence $Q_{K,k}$ is discordant.
\end{example}

%The joint generalization of Examples \ref{bfree} and \ref{kfree}---$b$-free integers in $\mathcal O_K$---gives more non-piecewise syndetic sets, but we are unable to make a statement regarding the density of these sets.

\begin{example}\label{byfree}
Given $k\in\N$ and a prime $p$, write $e_p(k)$ to denote the exponent on $p$ in the prime factorization of $k$.
In \cite[Theorem 3.9]{BBHS}, it is shown in a somewhat cumbersome way that, for any infinite set $P$ of primes and any corresponding collection of nonempty subsets $(Y_p)_{p\in P}$ of $\N$, the set \[A=\{k\in\N\colon\text{$e_p(k)\notin Y_p$ for each $p\in P$}\}\] is not piecewise syndetic.
This strengthens the non-piecewise syndeticity result of Example \ref{sqfree} and, in certain cases, Example \ref{bfree}, since it is clear that $A$ is a superset of the $(p^{\min(Y_p)})_{p\in P}$-free numbers. Below, we state this result in better generality and deduce it easily from Theorem \ref{bfreegeneral}.

Given any $b\in\N$ and $k\in\Z$, define $e_b(k)$ to be the largest power of $b$ dividing $k$ (set $e_b(0)=\infty$).
Now let $\mathscr{B} = (b_n)$ be a sequence of pairwise coprime positive integers and let $u = (u_n)$ be any sequence of positive integers. Define \[\mathcal{F}_{\mathscr{B}}^u=\{k\in\Z\colon\text{$e_{b_n}(k)\neq u_n$ for each $n\in\N$}\}.\]
Then $\mathcal{F}_{\mathscr{B}}^u\subseteq\Z\setminus\bigcup(b_n^{u_n+1}\N+b_n^{u_n})$, so $\mathcal{F}_{\mathscr{B}}^u$ is not piecewise syndetic by Theorem \ref{bfreegeneral}.
In Theorem \ref{byfree-density} below, we prove that, if $\sum b_n^{-u_n}<\infty$, then \[d(\mathcal{F}_{\mathscr{B}}^u)=\prod_{n\in\N}{\left(1-\frac{b_n-1}{b_n^{u_n+1}}\right)}.\]
\end{example}

\begin{example}\label{coprimepairs}
The set $C$ consisting of pairs $(a, b)$ of coprime integers can be expressed as
\[
C = (\Z \times \Z) \setminus \bigcup_{\text{$p$ prime}} (p\Z \times p\Z),
\]
making clear that $C$ is not piecewise syndetic by Theorem \ref{bfreegeneral}. It is well known that $d(C) = 6/\pi^2$, which will follow from the more general Theorem \ref{bbfree-density} (in other words, the probability that two ``randomly chosen'' integers are coprime is $6/\pi^2$; see, e.g., \cite[Theorem 332]{HaWr}).
Many possible generalizations in this direction can be imagined.
\end{example}

\begin{example}\label{heisenberg}
Our final example will take place in a noncommutative setting.
Consider the Heisenberg group $H_3(\Z)$, the matrix group
\[
{\left\{{\begin{pmatrix}
1 & a & c \\
0 & 1 & b \\
0 & 0 & 1
\end{pmatrix}}\colon a,b,c\in\Z\right\}}.
\]
Associating these matrices with triples $(a, b, c)$, we view $H_3(\Z)$ as $\Z^3$ with the group operation $(a_1, b_1, c_1) \cdot (a_2, b_2, c_2) = (a_1 + a_2, b_1 + b_2, a_1 b_2 + c_1 + c_2)$.
For each $n \in \N$, the kernel $n\Z^3$ of the entrywise reduction map $H_3(\Z) \to H_3(\Z/n\Z)$ is a normal subgroup of $H_3(\Z)$ of index $n^3$.
Hence, for any sequence $(b_n)$ of pairwise coprime positive integers, the set $H_3(\Z) \setminus \bigcup b_n\Z^3$ is not piecewise syndetic in $H_3(\Z)$ by Theorem \ref{bfreegeneral}.
We will see in Theorem \ref{heisenberg-density} that this set has density $\prod(1-1/b_n^3)$ with respect to an appropriate F\o lner sequence in $H_3(\Z)$.
\end{example}

\subsection{Computing density}

Before computing the density of some of the non-piecewise syndetic sets we constructed in Subsection \ref{nt-examples}, we make some definitions and prove a general lemma.

Suppose we have a countably infinite amenable cancellative semigroup $G$ with a F\o lner sequence $\Phi$.
Given a function $f \colon G \to \R$, we define
\[
\overline A_\Phi(f) = \limsup_{n \to \infty} \frac{1}{|\Phi_n|} \sum_{g \in \Phi_n} f(g),
\]
and likewise $\underline A_\Phi(f)$ and $A_\Phi(f)$ in the obvious way.
We will be most interested in the case when $A_\Phi(f)$ exists.
Assuming this limit exists, it has some obvious but useful properties.
First, given $E \subseteq G$, we have $A_\Phi(\mathbf1_E) = d_\Phi(E)$.
Second, if $f_1, f_2 \colon G \to \R$ satisfy $f_1 \leq f_2$, then $A_\Phi(f_1) \leq A_\Phi(f_2)$.
Finally, $f \mapsto A_\Phi(f)$ is \textit{finitely additive}:
$A_\Phi(f_1+f_2) = A_\Phi(f_1) + A_\Phi(f_2)$.

In order to prove that a set of the form $G \setminus \bigcup E_n$ has density, we will impose a number of nice properties on the family $(E_n)$.
We adopt the notation $E_I = \bigcap_{i \in I} E_i$ for $I \subseteq \N$. In addition, we use $\mathcal{P}_k(\N)$ to refer to the set of $k$-element subsets of $\N$.

Specifically, we say that $(E_n)$ is \textit{$\Phi$-inclusion-exclusion good}, or \textit{$\Phi$-I.E.\ good} if
\begin{enumerate}
\item
$d_\Phi(E_n)$ exists for each $n\in\N$, and $\sum_{n \in \N}d_\Phi(E_n)<\infty$;
\item
$E_I = \varnothing$ for every infinite $I \subseteq \N$;
\item
for every finite $I \subseteq \N$, $d_\Phi(E_I) = \prod_{i \in I}d_\Phi(E_i)$; and
\item
for each $k \in \N$, $A_\Phi(\sum_{I \in \mathcal{P}_k(\N)}\mathbf1_{E_I}) = \sum_{I \in \mathcal P_k(\N)}d_\Phi(E_I)$.
\end{enumerate}
For $\Phi$-I.E.\ good families, the following lemma states that an asymptotic version of the inclusion-exclusion principle holds.

\begin{lemma}\label{nicely-removable}
Suppose $G$ and $\Phi$ are as above and $(E_n)$ is a $\Phi$-I.E.\ good family of subsets of $G$. Write $\mathcal{F} = G \setminus \bigcup E_n$.
Then $d_\Phi(\mathcal{F}) = \prod(1 - d_\Phi(E_n))$.
\end{lemma}

\begin{proof}
By property (\textit{c}) of $\Phi$-I.E.\ goodness and the standard inclusion-exclusion principle, $\overline d_\Phi(\mathcal{F}) \leq \prod(1-d_\Phi(E_n))$, since $\mathcal{F} \subseteq G \setminus \bigcup_{i=1}^n E_i$ for each $n \in \N$.
Now let $n \in \N$.
Using the identity
\[
\prod_{n\in\N} (1-x_n)
=
\sum_{n=0}^{\infty} (-1)^n \sum_{I\in\mathcal{P}_n(\N)} \prod_{i\in I} x_i,
\]
which holds as long as $\sum x_i<\infty$,
properties (\textit{a}) and (\textit{c}) of $\Phi$-I.E.\ goodness yield
\[
\prod_{n\in\N}(1-d_\Phi(E_n))
=
\sum_{n=0}^{\infty}(-1)^n\sum_{I\in\mathcal{P}_n(\N)} d_\Phi(E_I).
\]
Thus $d_\Phi(\mathcal{F})=\prod(1-d_\Phi(E_n))$ will follow if we can demonstrate that
\[
\underline d_\Phi(\mathcal{F}) \geq 1 - \sum_{i\in\N}d_\Phi(E_i)+\sum_{I\in\mathcal{P}_2(\N)}d_\Phi(E_I)-\cdots-\sum_{I\in\mathcal P_{2n-1}(\N)}d_\Phi(E_I).
\]
To do this, it suffices, by the finite additivity of $f \mapsto A_\Phi(f)$ and property (\textit{d}) of $\Phi$-I.E.\ goodness, to show
\begin{equation}\label{lotsofindicators}
\mathbf1_\mathcal{F} \geq 1-\sum_{i\in\N}\mathbf1_{E_i}+\sum_{I\in\mathcal P_2(\N)}\mathbf1_{E_I}-\cdots-\sum_{I\in\mathcal P_{2n-1}(\N)}\mathbf1_{E_I}.
\end{equation}
Let $f$ denote the function on the right side of \eqref{lotsofindicators} and let $r \in \N$. If $r\in \mathcal{F}$, then clearly $f(r)=1$.
Otherwise, due to property (\textit{b}) of $\Phi$-I.E.\ goodness, we may suppose $r$ is contained in exactly $k\geq1$ elements of $(E_n)$.
Then
\[
f(r) = \sum_{i=0}^{2n-1}(-1)^i\binom{k}{i}=-\binom{k-1}{2n-1}\leq0,
\]
so indeed $\mathbf1_\mathcal{F} \geq f$.
This verifies that $\overline d_\Phi(\mathcal{F}) = \underline d_\Phi(\mathcal{F}) = \prod(1-d_\Phi(E_n))$.
\end{proof}

Using this lemma, we compute the density of sets constructed in Examples \ref{bfree} and \ref{byfree}.
The following two theorems use the standard F\o lner sequence $(\{1,\dots,n\})_{n\in\N}$ in $\N$, which (as with the notation $d$) we omit reference to.

\begin{theorem}\label{bfree-density}
Let $\mathscr{B} = (b_n)$ be a sequence of pairwise coprime positive integers satisfying $\sum1/b_n<\infty$.
Then $d(\mathcal{F}_\mathscr{B}) = \prod(1-1/b_n)$.
\end{theorem}

\begin{proof}
We prove this statement for $\mathcal{F}_\mathscr{B} = \N\setminus\bigcup b_n\N$ for convenience, but it is of no consequence, since $\mathcal{F}_\mathscr{B}$ is symmetric about $0$.
We must show that $(b_n\N)$ is I.E.\ good (with respect to the standard F\o lner sequence as mentioned above).
Only property (\textit{d}) is not immediate.
Notice that it suffices to prove the case $k=1$, since any intersection of sets of the form $b\Z$ again has this form.

Write $f=\sum\mathbf1_{b_n\N}$.
Since $\mathbf1_{b_1\N} \leq \mathbf1_{b_1\N}+\mathbf1_{b_2\N}\leq\cdots\leq f$, clearly $\underline A(f)\geq\sum1/b_n$.
To prove $\overline A(f)\leq\sum1/b_n$, the key point is that, for each $k,n\in\N$,
\[
\frac1k\sum_{i=1}^k\mathbf1_{b_n\N}(i)\leq\frac1{b_n},
\]
which is just a statement of the fact that there are no more than $k/b_n$ multiples of $b_n$ in $\{1, \dots, k\}$.
It follows that, for each $k\in\N$,
\[
\frac{1}{k}\sum_{i=1}^kf(i)\leq\sum_{n\in\N}\frac1{b_n},
\]
whence $\overline A(f) = \underline A(f) = \sum1/b_n$. Thus $(b_n\N)$ is I.E.\ good.
\end{proof}

\begin{theorem}\label{byfree-density}
Let $\mathscr{B} = (b_n)$ be a sequence of pairwise coprime positive integers and let $u = (u_n)$ be a sequence of positive integers for which $\sum b_n^{-u_n}<\infty$.
Then
\[
d(\mathcal{F}_\mathscr{B}^u) = \prod_{n\in\N}{\left(1-\frac{b_n-1}{b_n^{u_n+1}}\right)}.
\]
\end{theorem}

\begin{proof}
As in Theorem \ref{bfree-density} we confine ourselves to $\N$.
For each $n \in \N$ define $B_n = \{k\in\N\colon e_{b_n}(k)=u_n\}$, so that $\mathcal{F}_\mathscr{B}^u \cap \N = \N \setminus \bigcup B_n$.
We aim to show that $(B_n)$ is an I.E.\ good family.
(The density $d(B_n)$ is easy to compute but is also derived below.)
Clearly property (\textit{b}) of I.E.\ goodness is satisfied. %, and since $d(B_n) = (b_n-1)/(b_n^{u_n+1})$ for each $n \in \N$, so is (\textit{a}) by assumption.

Given $I\subseteq \N$, define also $B_I = \bigcap_{i \in I}B_i$, $b_I = \prod_{i \in I} b_i$, and $r_I = \prod_{i \in I} b_i^{u_i}$.
Observe that $B_I \subseteq r_I\N$.
Define $C_I = B_I/r_I$ to be the set of positive integers not divisible by $b_i$ for any $i \in I$.
Observe that $C_I$ is $b_I$-periodic, i.e., for each $k \in \N$, one has $k \in C_I$ if and only if $b_I + k \in C_I$.
In particular, $d(C_I) = |C_I \cap \{1, \dots, b_I\}|/b_I$.
Now consider the isomorphism $\phi\colon\Z/b_I\Z \overset{\sim}{\to} \prod_{i \in I} \Z/b_i\Z$ afforded by the classical Chinese remainder theorem.
An element $x \in \Z/b_I\Z$ is in $C_I$ if and only if $\phi(x)$ is nonzero in every coordinate, from which it follows that $|C_I \cap \{1, \dots, b_I\}| = \prod_{i \in I} (b_i - 1)$.
From this we get
\[
d(B_I) =
d(r_IC_I) =
\frac{1}{b_Ir_I} \prod_{i \in I} (b_i - 1)
=
\prod_{i \in I} \frac{b_i - 1}{b_i^{u_i + 1}}
=
\prod_{i \in I} d(B_i),
\]
proving that property (\textit{c}) above holds for $(B_n)$.
In addition, the formula we have obtained for $d(B_i)$ and the hypotheses on $\mathscr{B}$ and $u$ verify property (\textsl{a}).

% Given $k \in \N$, define the set
% \[
% C(k) = \{\ell \in \N\colon\text{for each $n\in\N$, it is false that $b_n$ divides both $k$ and $\ell$}\}.
% \]
% Note that $C(k)$ is $k$-periodic, i.e., for each $\ell\in\N$, $k+\ell\in C(k)$ if and only if $\ell\in C(k)$.
% Let $c(k) = |C(k)\cap\{1,\dots,k\}|$, so that $d(C(k))=c(k)/k$.

% Now given a finite set $I \subseteq \N$, define $r_I=\prod_{i\in I}b_i^{y_i}$ and $s_I=\prod_{i\in I}b_i$.
% Then it is easily shown that $B_I=r_IC(s_I)$. It follows from the classical Chinese remainder theorem that $c(b_ib_j)=c(b_i)\cdot c(b_j)$ for distinct $i, j\in\N$, from which we deduce that $d(B_I)=\prod_{i\in I}d(B_i)$. This proves property (c) above.

Now fix $k\in\N$ and set $f=\sum_{I\in\mathcal P_k(\N)}\mathbf 1_{B_I}$. As in Theorem \ref{bfree-density} we have $\underline A(f)\geq\sum d(B_I)$.
To prove that $\overline A(f)\leq\sum d(B_I)$, fix $\varepsilon>0$. Select $N\in\N$ such that
\[
\sum_{\substack{I\in\mathcal P_k(\N)\\\max(I)>N}}\frac{1}{r_I}<\frac{\varepsilon}{2}.
\]
Now let $I\in\mathcal P_k(\N)$ and note that, since $B_I=r_IC_I$ where $C_I$ is $b_I$-periodic, we have, for each $a\in\N$ and $m\geq ab_Ir_I$,
\[
\frac{|B_I\cap\{1,\dots,m\}|}{m}\leq\frac{(a+1)\cdot|B_{I}\cap\{1,\dots,b_Ir_I\}|}{ab_Ir_I}=\frac{a+1}{a}\cdot d(B_I).
\]
Find $a\in\N$ for which $\sum_{I\in\mathcal P_k(\N)}d(B_I)/a<\varepsilon/2$ and set $R=\max_{I\in\mathcal P_k(\{1,\dots,N\})}b_Ir_I$. Then, for each $m\geq aR$,
\begin{align*}
\frac{1}{m}\sum_{k=1}^mf(k)&=\sum_{I\in\mathcal P_k(\{1,\dots,N\})}\frac{|B_I\cap\{1,\dots,m\}|}{m}+\sum_{\substack{I\in\mathcal P_k(\N)\\\max(I)>N}}\frac{|B_I\cap\{1,\dots,m\}|}{m}\\
&\leq\sum_{I\in\mathcal P_k(\{1,\dots,N\})}\frac{a+1}{a}\cdot d(B_I)+\sum_{\substack{I\in\mathcal P_k(\N)\\\max(I)>N}}\frac{1}{r_I}\\
&<\left(\sum_{I\in\mathcal P_k(\N)}d(B_I)\right) + \frac{\varepsilon}{2}+\frac{\varepsilon}{2}
\end{align*}
(we have $|B_I \cap \{1, \dots, m\}|/m \leq 1/r_I$ since $B_I \subseteq r_I\N$).
This verifies that $(B_n)$ is indeed I.E.\ good and completes the proof.
\end{proof}

We next compute the density of the sets constructed in Examples \ref{coprimepairs} and \ref{heisenberg}. To do so, we will need a technical lemma.

\begin{lemma}
\label{highdimremovability}
Let $G$ and $H$ be  countably infinite  amenable cancellative semigroups with F\o lner sequences $\Phi$ and $\Psi$, respectively. Let $(D_n)$ be a $\Phi$-I.E.\ good family of subsets of $G$ and $(E_n)$ be a $\Psi$-I.E.\ good family of subsets of $H$. Then $(D_n\times E_n)$ is a $(\Phi\times\Psi)$-I.E.\ good family of $G\times H$, where $\Phi \times \Psi = (\Phi_n \times \Psi_n)_{n \in \N}$.
\end{lemma}

\begin{proof}
The fact that $\Phi\times\Psi$ is a F\o lner sequence for $G\times H$ is easy and left to the reader.
We outline the proof that $(D_n\times E_n)$ satisfies property (\textit{d}) of $(\Phi \times \Psi)$-I.E.\ goodness; the others are straightforward.
We have, for each $k \in \N$,
\[
A_{\Phi\times\Psi}{\left(\sum_{I\in\mathcal{P}_k(\N)}\mathbf1_{D_I\times E_I}\right)}
=
\lim_{n\to\infty}\sum_{I\in\mathcal{P}_k(\N)}\frac{|D_I\cap\Phi_n|}{|\Phi_n|}\cdot\frac{|E_I\cap\Psi_n|}{|\Psi_n|}.
\]
But since
\[
A_\Phi{\left(\sum_{I \in \mathcal{P}_k(\N)}\mathbf{1}_{D_I}\right)}
=
\lim_{n \to \infty} \sum_{I \in \mathcal{P}_k(\N)} \frac{|D_I \cap \Phi_n|}{|\Phi_n|}
=
\sum_{I \in \mathcal{P}_k(\N)} d_\Phi(D_I),
\]
and likewise with $E_I$, proving property (\textit{d}) reduces to verifying the following claim:
given doubly-indexed sequences $(a_{k,n})$ and $(b_{k,n})$ of nonnegative quantities satisfying
\begin{itemize}
\item
$\lim_{n \to \infty} a_{k, n} = a_k$ and $\lim_{n \to \infty} b_{k, n} = b_k$ for each $k$, and
\item
$\lim_{n \to \infty} \sum_{k \in \N} a_{k, n} = \sum a_k < \infty$ and $\lim_{n \to \infty} \sum_{k \in \N} b_{k, n} = \sum b_k < \infty$,
\end{itemize}
one has $\lim_{n \to \infty} \sum_{k \in \N} a_{k,n} b_{k,n} = \sum a_k b_k$.
\end{proof}

\begin{theorem}\label{bbfree-density}
Let $d \geq 2$ and let $(b_{n,1}), \dots, (b_{n,d})$ be sequences such that, for each $i$, $(b_{n,i})$ consists of pairwise coprime positive integers.
Then
\[
d{\left(\Z^d \setminus \bigcup_{n \in \N} (b_{n,1}\Z\times\cdots\times b_{n,d}\Z)\right)} = \prod_{n \in \N}{\left(1 - \frac{1}{b_{n,1} \cdots b_{n,d}}\right)}.
\]
\end{theorem}

\begin{proof}
This follows from Lemma \ref{nicely-removable} and Theorem \ref{bfree-density} via repeated applications of Lemma \ref{highdimremovability}.
\end{proof}

To apply Lemma \ref{highdimremovability} to the Heisenberg group $H_3(\Z)$, viewed as $\Z^3$ with the group operation described in Example \ref{heisenberg}, a laborious but straightforward computation shows that
\begin{equation}\label{heisenbergfolner}
\Phi = (\{-n, \dots, n\} \times \{-n^2, \dots, n^2\} \times \{-n, \dots, n\})_{n \in \N}
\end{equation}
is a F\o lner sequence in $H_3(\Z)$.

\begin{theorem}\label{heisenberg-density}
Let $(b_n)$ be a sequence of pairwise coprime positive integers.
Then, letting $\Phi$ be as in \eqref{heisenbergfolner}, we have
\[
d_\Phi{\left(H_3(\Z) \setminus \bigcup_{n \in \N} b_n\Z^3\right)} = \prod_{n \in \N} {\left(1-\frac{1}{b_n^3}\right)}.
\]
\end{theorem}

\begin{proof}
We may again apply Lemma \ref{nicely-removable}, Theorem \ref{bfree-density}, and Lemma \ref{highdimremovability} to deduce that $\Z^3\setminus\bigcup b_n\Z^3$ is $\Phi$-I.E.\ good as a subset of $(\Z^3,+)$ (notice that $\Phi$-I.E.\ goodness is unaffected by passing to a subsequence of $\Phi$---in our case from $(\{-n,\dots,n\})_{n\in\N}$ to $(\{-n^2,\dots,n^2\})_{n\in\N}$).
Since the definition of density does not depend on the group operation, we are done.
\end{proof}

\section{Discordant sets in \texorpdfstring{$\SL_2(\Z)$}{SL(2,Z)}}\label{curious-settings}

Let $G$ denote the free semigroup generated by $a$ and $b$ with identity $1\in G$. It is well known that $G$ is not amenable. To see why, suppose $\Phi$ is a F\o lner sequence for $G$, and form the partition $G=\{1\}\cup aG\cup bG$. One of these parts must have positive upper density with respect to $\Phi$, and since $\{1\}$ does not, suppose without loss of generality that $c=\overline d_\Phi(aG)>0$. Then $aG,baG,b^2aG,\dots$ are disjoint subsets of $G$ each with upper density $c>0$. But Lemma \ref{shift-additivity} guarantees that this is impossible, which is our desired contradiction.
This argument can be modified without too much difficulty to show that nonabelian free groups are not amenable, and moreover any group possessing a nonabelian free subgroup is not amenable.

We now turn our attention to the group $\SL_2(\mathbb{Z})$. Since $\SL_2(\mathbb{Z})$ contains nonabelian free subgroups,\footnote
{For example, it is well-known that the matrices $\smatr{1&2\\0&1}$ and $\smatr{1&0\\2&1}$ generate a free subgroup of $\SL_2(\Z)$ (see \cite{Sa}).
}
it is not amenable. However, we will define a natural density-like notion for subsets of $\SL_2(\Z)$ and construct sets which are large with respect to this notion but not piecewise syndetic.
Although this density does not enjoy shift-invariance (which in the amenable case is assured by F\o lner sequences), we consider this result an interesting aside which has independent value.
In particular, it makes more robust the analogy with nowhere dense sets of positive measure, showing that it extends beyond the amenable case.

We first leverage Theorem \ref{bfreegeneral} to produce an ample supply of non-piecewise syndetic subsets of $\SL_2(\mathbb{Z})$.
Recall the \textit{congruence subgroups} of $\SL_2(\Z)$, which are subgroups of the form
\[
\Gamma(n)={\left\{\matr{a&b\\c&d}\in\SL_2(\Z)\colon\matr{a&b\\c&d}\equiv\matr{1&0\\0&1}\pmod n\right\}}.
\]
The subgroup $\Gamma(n)$ is normal in $\SL_2(\Z)$ for each $n\in\N$ since it is the kernel of the homomorphism $\phi_n\colon\SL_2(\Z)\to\SL_2(\Z/n\Z)$ defined by reducing each matrix entrywise modulo $n$.

\begin{theorem}\label{sl2z}
Let $(b_n)$ be a sequence of pairwise coprime positive integers, and define $A=\SL_2(\Z)\setminus\bigcup\Gamma(b_n)$. Then $A$ is not piecewise syndetic.
\end{theorem}

\begin{proof}
Fix any coprime $m,n\in\N$. Since $\Gamma(m)\cap\Gamma(n)=\Gamma(mn)$, to apply Theorem \ref{bfreegeneral} it suffices to show that $\Gamma(m)\Gamma(n)=\SL_2(\Z)$. We claim that we may always find, for any $T\in\SL_2(\Z/m\Z)$, some $S\in\SL_2(\Z)$ such that $\phi_m(S)=T$ and $\phi_n(S)=I$.
Granted this, fix $S\in\SL_2(\Z)$, and find $S_1,S_2\in\SL_2(\Z)$ such that $\phi_m(S_1)=\phi_m(S)$, $\phi_n(S_1)=I$, $\phi_m(S_2)=I$, and $\phi_n(S_2)=\phi_n(S)$.
Then $SS_1^{-1}S_2^{-1}\in\ker(\phi_m)\cap\ker(\phi_n)=\Gamma(mn)$, so there exists $T\in\Gamma(mn)$ such that $S=(TS_2)S_1\in\Gamma(m)\Gamma(n)$.
Hence $\Gamma(m)\Gamma(n)=\SL_2(\Z)$, as desired.

To prove the claim, let $T=\smatr{a&b\\c&d}$, treating the entries as elements of $\Z$. Find $x,y\in\Z$ such that $mx+ny=1$ and set
\[
S'=\matr{a'&b'\\c'&d'}=\matr{mx+any&bny\\cny&mx+dny},
\]
so $S'\equiv T\pmod m$ and $S'\equiv I\pmod n$.
Note that $\det(S')\equiv1\pmod {mn}$.
It follows that $\gcd(a',b',mn)=1$.
Letting $q$ denote the product of primes dividing $a'$ but not $b'$, we see that $a'$ and $b'+mnq$ are coprime.
Therefore, there exist $u,v\in\Z$ such that $a'u+(b'+mnq)v=1$.
It also follows from $\det(S')\equiv1 \pmod{mn}$ that there exists $s\in\Z$ such that
\[
\det\matr{a'&b'+mnq\\c'&d'}=mns+1.
\]
Then
\[
S=\matr{a'&b'+mnq\\c'+mnsv&d'-mnsu}
\]
is an element of $\SL_2(\Z)$ satisfying $\phi_m(S)=T$ and $\phi_n(S)=I$.
\end{proof}

We will now introduce a notion of density for subsets of $\SL_2(\mathbb{Z})$. Define \[F_n={\left\{\begin{pmatrix}a & b \\ c & d\end{pmatrix}\in\SL_2(\Z)\colon|a|,|b|,|c|,|d|\leq n\right\}}.\] We then define the \emph{upper density} of $A\subseteq \SL_2(\mathbb{Z})$ to be \[\overline{d}(A)=\limsup_{n\to\infty}\frac{|A\cap F_n|}{|F_n|},\]
and define $\underline{d}$ and $d$ analogously.
We first make some crude estimates on the growth of the sets $F_n$ and $\Gamma(k) \cap F_n$.

\begin{lemma}\label{estim1}
There is a positive integer $N$ such that $|F_n|\geq(12/\pi^2)n^2$ for all $n\geq N$. 
\end{lemma}

\begin{proof}
We will produce, for sufficiently large $n$, at least $(12/\pi^2)n^2$ matrices $\smatr{a & b \\ c & d} \in \SL_2(\mathbb{Z})$ such that $\max\{|a|,|b|,|c|,|d|\}\in\{|a|,|b|\}$ and $|a|,|b|,|c|,|d|\leq n$. As we saw in Example \ref{coprimepairs}, \[\lim_{n\to\infty}\frac{|\{a,b\in\mathbb{Z}\colon\text{$|a|,|b|\leq n$ and $\gcd(a,b)=1$}\}|}{(2n+1)^2}=\frac{6}{\pi^2}.\] Thus there is a positive integer $N$ such that for all $n\geq N$, there are at least $(3/\pi^2)(2n+1)^2>(12/\pi^2)n^2$ pairs $(a,b)$ such that $\max\{|a|,|b|\}\leq n$ and $\gcd(a,b)=1$.
Given such a pair $(a,b)$, using the Euclidean algorithm, we can find $c,d \in \Z$ with $|c|\leq |a|$ and $|d|\leq |b|$ such that $ad-bc=1$.
Then $\smatr{a & b \\ c & d}\in F_n$, as desired. 
\end{proof}

\begin{lemma}\label{estim2}
For each $k\geq 2$ and $n\geq k$, we have $|\Gamma(k)\cap F_n|\leq (96/k^2)n^2$.\end{lemma}

\begin{proof}
Define \[\Lambda(k)={\left\{\begin{pmatrix}a & b \\ c & d \end{pmatrix}\in \Gamma(k)\colon\max\{|a|,|b|,|c|,|d|\}\in \{|a|,|b|\}\right\}}.\]
Given $M\in \Gamma(k)\cap F_n$, either $M$ or the matrix given by swapping the rows  and columns of $M$ is in $\Lambda(k)\cap F_n$. Thus $|\Gamma(k)\cap F_n|\leq 2\cdot|\Lambda(k)\cap F_n|$.

There are at most $\lceil(2n+1)/k\rceil^2 \leq ((2n+1+k)/k)^2 \leq (4n/k)^2$ pairs $(a,b) \in \Z^2$ with $|a|,|b|\leq n$ such that $a\equiv1\pmod k$ and $b\equiv0\pmod k$. In particular, there are at most $(4n/k)^2$ such pairs $(a, b)$ with $\gcd(a,b)=1$.
Given such a pair, one can find, using the Euclidean algorithm, $c,d \in \Z$ such that $ad-bc=1$, $|d|\leq |b|$, $|c|\leq |a|$.
Then
\[\{(x,y)\in \mathbb{Z}^2\colon ax-by=1\}=\{(d+mb,c+ma)\colon m\in\mathbb{Z}\},\]
so there are at most three pairs $(x,y)$ such that $ax-by=1$ and $\max\{|x|,|y|\}\leq \max\{|a|,|b|\}$, and in particular there are at most three $(x,y)$ such that $\smatr{a & b \\ y & x}\in \Lambda(k)\cap F_n$.

Combining all of this, we have $|\Gamma(k)\cap F_n|\leq 6(4n/k)^2 = (96/k^2)n^2$. 
\end{proof}

With these estimates, we can show that under certain conditions, the sets in Theorem \ref{sl2z} have positive lower density.

\begin{theorem} Let $(b_n)$ be a sequence of pairwise coprime positive integers such that $\sum 8\pi^2/b_n^2<1$. Then $A=\SL_2(\mathbb{Z})\setminus \bigcup\Gamma(b_n)$ has positive lower density.
\end{theorem}

\begin{proof} For $2\leq n< b_m$, the only element of $\Gamma(b_m)\cap F_n$ is the identity matrix $I$. For $m\geq 2$, write $\Gamma'(b_m)=\Gamma(b_m)\setminus \{I\}$ and $\Gamma'(b_1)=\Gamma(b_1)$, so that $\bigcup\Gamma(b_m)=\bigcup\Gamma'(b_m)$.
Then, for $2\leq n< b_m$, we have $\Gamma'(b_m)\cap F_n=\varnothing$, so \[A\cap F_n=F_n\setminus\bigcup_{b_m\leq n}(\Gamma'(b_m)\cap F_n).\] %hich gives \[\frac{|A\cap F_n|}{|F_n|}\geq 1-\sum_{b_m\leq  n}\frac{|\Gamma(b_n)\cap F_n|}{|F_n|}.\]
For sufficiently large $n$, we then have, applying Lemmas \ref{estim1} and \ref{estim2}, \[\frac{|A\cap F_n|}{|F_n|}\geq
1-\sum_{b_m\leq  n}\frac{|\Gamma'(b_m)\cap F_n|}{|F_n|} \geq
1-\sum_{b_m\leq n}\frac{(96/b_m^2)n^2}{(12/\pi^2)n^2}
%=1-\sum_{b_m\leq  n}\frac{C_{b_m}}{C}
\geq 1-\sum_{m\in\mathbb{N}}\frac{8\pi^2}{b_m^2}.\]
Thus $\underline{d}(A)\geq 1-\sum 8\pi^2/b_n^2>0$, as desired.
\end{proof}

\section{Recurrence and anti-recurrence}\label{recurrence}
%One of the first examples of discordant sets was given by Straus, answering in the negative a question by Erd\H{o}s, who asked whether any set of positive upper density in $\mathbb{Z}$ must contain a shift of an IP set. Since any piecewise syndetic set must contain a shift of an IP set, such an example gives rise to a discordant set. (See \cite{BCRZ} for further discussion.) 

% A set $E\subseteq\mathbb{N}$ is a \textit{set of recurrence} if for all finite measure preserving systems $(X,\mathcal{B},\mu,T)$ and for all $A\in\mathcal{B}$ with positive measure, there exists $r\in E$ with $\mu(A\cap T^{-r}A)>0$. Through Furstenberg's correspondence principle, one can see that equivalently, $E\subseteq\mathbb{N}$ is a set of recurrence if and only if for all $A\subseteq \mathbb{N}$ with $\overline{d}(A)>0$, $\overline{d}(E\cap (A-A))>0$. Such sets are also called \textit{intersective sets} in the literature. 

%is this the moral general definition of set of recurrence? who's to say
Let $G$ be a  countably infinite  group. A set $A\subseteq G$ is a \textit{set of \rmpars{measurable} recurrence} if for all measure preserving systems $(X,\mathcal{B},\mu,G)$ and for all $Y\subseteq X$ with $\mu(Y)>0$, there is a $g\in A$ not equal to the identity such that $\mu(Y\cap gY)>0$. Recall that in a measure-preserving system, we assume $\mu(X)=1$, so sets of recurrence exist (for example, take $A = G$).

% Call a set $A\subseteq \mathbb{N}$ an \textit{anti-recurrence} set, or an AR set, if for all $n\in\mathbb{N}$, $A-n$ is not a set of recurrence. 

Call $A\subseteq G$ an \textit{anti-recurrence set}, or an AR set, if for all $g\in G$, $gA$ is not a set of recurrence. We will begin by showing that any piecewise syndetic set is a shift of a set of recurrence, so any AR set with positive upper density is therefore discordant.

\begin{lemma}
Let $T \subseteq G$ be thick.
Then $T$ is a set of recurrence.
\end{lemma}

\begin{proof}
Let $T\subseteq G$ be thick. First, we will show that there is an infinite sequence $(g_n)$ of elements of $G$ such that for all pairs $i,j\in\N$ with $i<j$, we have $g_i^{-1}g_j\in T$. Indeed, such a set can be constructed recursively: given $g_1,\dots,g_n$, we can find $g_{n+1}$ such that $\{g_1^{-1},g_2^{-1},\dots,g_n^{-1}\}g_{n+1}\subseteq T$. 

Let $(X,\mathcal{B},\mu,G)$ be a measure preserving system, and let $Y\subseteq X$ have $\mu(Y)>0$. There exist $g_i,g_j$ with $i<j$ such that $\mu(g_iY\cap g_jY)>0$. Then, $\mu(Y\cap g_i^{-1}g_j Y)>0$, as desired.
\end{proof}

\begin{lemma}
Let $A\subseteq G$ be a set of recurrence partitioned as $A=\bigcup_{i=1}^r A_i$.
Then at least one $A_i$ is a set of recurrence.
\end{lemma}

\begin{proof}
Assume that no $A_i$ is a set of recurrence. Then, for each $i$, there is a measure preserving system $(X_i,\mathcal{B}_i,\mu_i,G)$ and $Y_i\subseteq X_i$ with $\mu_i(Y_i)>0$ such that for all $g\in A_i$ we have $\mu_i(Y_i\cap gY_i)=0$. The actions of $G$ on each $X_i$ induce an action on $X=X_1\times\dots\times X_n$, which is measure preserving with respect to the product measure $\mu$ of the measures $\mu_i$. Let $Y=Y_1\times\dots\times Y_n$. We have $\mu(Y_1\times\dots\times Y_n)>0$, but for all $g\in A$, since $g\in A_i$ for some $i$, we have $\mu_i(Y_i\cap gY_i)=0$, and thus $\mu(Y\cap gY)=0$. Therefore, $A$ is not a set of recurrence.
This establishes the contrapositive.
\end{proof}

\begin{theorem} If $A$ is AR, then $A$ is not piecewise syndetic.
\end{theorem}
\begin{proof} Apply Lemma \ref{PS-friendly} with $\mathcal{A} = \{gA\colon A\text{ is a set of recurrence}\}$, noting the previous two lemmas.
\end{proof}

In \cite{BCRZ}, it is noted that Straus' example of a discordant set is also an AR set. We will follow their terminology and refer to AR sets with positive upper density as \text{Straus sets}. In what follows, we investigate the question of whether a given discordant set is a Straus set. We begin by showing that Example \ref{bfree} provides large class of non-examples.

\begin{theorem}\label{bfreenotstrauss}
Let $\mathscr{B} = (b_n)$ be a sequence of pairwise coprime positive integers such that $\sum1/b_n<\infty$.
Define $\mathcal{F}_\mathscr{B}=\Z\setminus\bigcup b_n\Z$.
Then $\mathcal{F}_\mathscr{B}$ is a discordant set which is not AR.
\end{theorem}

\begin{proof}
As noted in Example \ref{bfree}, $\mathcal{F}_\mathscr{B}$ is discordant.
Now let $b\in \mathcal{F}_\mathscr{B}$.
By \cite[Theorem 2.8]{BR}, $\mathcal{F}_\mathscr{B}-b$ contains an IP set.
Since IP sets are sets of recurrence, $\mathcal{F}_\mathscr{B}-b$ is also a set of recurrence, and hence $\mathcal{F}_\mathscr{B}$ cannot be an AR set.
\end{proof}

%In particular, for any $k\in\N$, the set of $k$-free numbers is a discordant set which is not AR.

Not all  countably infinite  amenable groups admit Straus sets. For example, in the group of finitely supported even permutations of $\mathbb{N}$, every set with positive upper Banach density is a set of recurrence (in fact, such a set must be IP).
In \cite[Theorem 1.23]{BCRZ}, it is shown that Straus sets exist in a locally compact, second countable, amenable group if and only if the von Neumann kernel of that group is not cocompact.
In such groups, for all $\varepsilon>0$, there is a Straus set $E$ with $d(E)>1-\varepsilon$. Thus, it is natural to ask whether such a group has a Straus set $E$ with $d(E)=c$ for each $c\in(0,1)$. We prove that this is the case for all  countably infinite  abelian groups in Theorem \ref{BetterARall}. First, we need the following lemma, a special case of a lemma appearing in \cite{BCRZ}. 

\begin{lemma}[{cf.\ \cite[Lemma 5.1]{BCRZ}}]\label{fivepointone} Let $G$ be a  countably infinite  amenable group and let $\Phi$ be a F\o lner sequence in $G$. Let $(A_n)$ be a sequence of subsets of $G$ such that $d_{\Phi}(A_n)$ exists for every $n$, $\sum d_{\Phi}(A_n)<\infty$, and $d_{\Phi}(A_1\cup\dots \cup A_n)$ exists for every $n$. Then there are cofinite subsets $A_n'\subseteq A_n$ such that for $C=\bigcup A_n'$, $d_{\Phi}(C)=\lim d_{\Phi}(A_1\cup\dots\cup A_n)$.
\end{lemma}
% \begin{proof} For each $i\in\mathbb{N}$, there is an $s_i\in\mathbb{N}$ such that $\left|\frac{|A_i\cap \Phi_n|}{|\Phi_n|}-d_{\Phi}(A_i)\right|<\frac{1}{2^i}$ whenever $N\geq s_i$. We may assume $s_{i+1}>s_i$ for all $i\in\mathbb{N}$. We will show that the conclusion holds for $A_i'=A_i\setminus (\Phi_1\cup \dots \cup \Phi_{s_i})$. Note that $\frac{|A_i'\cap \Phi_n|}{|\Phi_n|}<d_{\Phi}(A_i)+\frac{1}{2^i}$ for every $N$.

% Define $B_n=A_1\cup\dots\cup A_n$ and $C_n=A_1'\cup\dots\cup A_n'$. Let $c=\lim d_{\Phi}(B_n)$. Since $C_n$ is a cofinite subset of $B_n$ we have $d_{\Phi}(C_n)=d_{\Phi}(B_n)$. Since $C_n\subseteq C_{n+1}$, we have $\underline{d}_{\Phi}(C)\geq c$. It remains to show that $\overline{d}_{\Phi}(C)\leq c$. For every $J$ and every $N$, we have \[\frac{|C\cap \Phi_N|}{|\Phi_N|}\leq \frac{|C_J\cap \Phi_N|}{|\Phi_N|}+\sum_{i=J+1}^\infty \frac{|A_i'\cap \Phi_N|}{|\Phi_N|}\]\[\leq \frac{|B_J\cap \Phi_N|}{|\Phi_N|}+\sum_{i=J+1}^\infty d_{\Phi}(A_i)+\frac{1}{2^i}.\]
% Letting $N\to\infty$ we obtain \[\overline{d}_{\Phi}(C)\leq d_{\Phi}(B_J)+2^{-J}+\sum_{i=J+1}^\infty d_{\Phi}(A_i).\]
% Letting $J\to\infty$, we obtain $\overline{d}_{\Phi}(C)\leq c$, as desired. 
% \end{proof}

Before proving Theorem \ref{BetterARall}, we give a proof of the case $G=\mathbb{Z}$ which does not rely on tools from ergodic theory. This also serves as a model of the proof of Theorem \ref{BetterARall} for arbitrary discrete  countably infinite  abelian groups.

For the proof of Theorem \ref{ARall}, the notion of \textit{well-distribution} for a sequence in $\T$ will be used. For a more detailed exposition of this theory, see \cite{KN}. 
A sequence $(x_n)$ in $\T$ is called \textit{well-distributed} if for any measurable $A\subseteq \T$ with $\mu(\partial A)=0$, \[\lim_{N-M\to\infty}\frac{1}{N-M}\sum_{n=M+1}^N\mathbf{1}_{A}(x_n)=\mu(A),\]
where $\mu$ is the Lebesgue measure and $\partial A$ is the boundary of $A$, i.e., the closure of $A$ minus the interior of $A$. 
Equivalently, $(x_n)$ is well-distributed if for any F\o lner sequence $\Phi$ in $\Z$ and any measurable $A\subseteq\T$ with $\mu(\partial A)=0$, \[\lim_{n\to\infty}\frac{1}{|\Phi_n|}\sum_{k\in \Phi_n}\mathbf{1}_A(x_k)=\mu(A).\]

\begin{theorem}\label{ARall} Let $c\in(0,1)$ and let $\Phi$ be a F\o lner sequence in $\mathbb{Z}$. Then, there is a Straus set $A\subseteq \mathbb{Z}$ such that $d_{\Phi}(A)=c$.
\end{theorem}

\begin{proof}
Let $t\in(0,1)$, and let $\alpha \in \T$ be irrational. For a set $A\subseteq \T$, let $R_A$ denote $\{n\in\mathbb{Z}\colon n\alpha\in A\}$. For $n\geq 0$, let $r_n=t/((2n+1)\cdot2^{n+1})$, let \[B_n=\bigcup_{k=-n}^n (-r_n+k\alpha,r_n+k\alpha),\] and let $C_n=\bigcup_{k=0}^n B_{k}$. Since each of $B_n$ and $C_n$ have boundaries with measure zero, and since $(n\alpha)$ is well-distributed in $\T$, we have $d_{\Phi}(R_{B_n})=\mu(B_n)$ and $d_{\Phi}(R_{C_n})=\mu(C_n)$ for all $n$. Since $\sum\mu(B_n)\leq 2t$, by Lemma \ref{fivepointone}, there are cofinite subsets $R'_{B_n}\subseteq R_{B_n}$ such that
\[d_{\Phi}{\left(\bigcup_{n=0}^\infty R'_{B_n}\right)}=\lim_{n\to\infty}\mu(C_n)=\mu{\left(\bigcup_{n=0}^\infty B_n\right)}.\]
Let $f_n(t)=\mu(C_n)$, and let $f(t)=\lim_{n\to\infty}f_n(t)$. For all $t$, \[|f(t)-f_n(t)|\leq \sum_{k=n+1}^\infty \mu(B_{k})\leq \frac{1}{2^n},\] so the convergence is uniform.
Since the boundary of each interval has measure zero, the functions $f_n$ are continuous, so $f$ is continuous. Moreover, $\lim_{t\to 1}f(t)=1$ (since $\lim_{t\to 1}\mu(B_0) = 1$) and $\lim_{t\to 0}f(t)=0$, so we can choose $t$ so that $d_{\Phi}\left(\bigcup R_{B_n}'\right)=1-c$. Let $A=\mathbb{Z}\setminus \bigcup R'_{B_n}$. Then, $d_{\Phi}(A)=c$. It remains to show that $A$ is AR. Let $k\in\Z$.
Then \[(A-k)\cap{\left \{n\in\mathbb{Z}\colon{\left(\frac{-r_k}{2}+n\alpha,\frac{r_k}{2}+n\alpha\right)}\cap {\left(\frac{-r_k}{2},\frac{r_k}{2}\right)}\neq\varnothing\right\}}\] is finite, so $A-k$ cannot be a set of recurrence.  
\end{proof}

In fact, a version of Theorem \ref{ARall} holds true for all  countably infinite  abelian groups.

\begin{theorem}\label{BetterARall} Let $G$ be a \rmpars{discrete}  countably infinite  abelian group. Let $c\in(0,1)$, and let $\Phi$ be a F\o lner sequence in $G$. Then, there is a Straus set $A\subseteq G$ such that $d_{\Phi}(A)=c$.
\end{theorem}
\begin{proof} We will construct a compact abelian   group $X$ and a homomorphism $\phi\colon G\to X$ with dense image. Depending on $G$, this construction can occur three different ways.

First, assume that $G$ has an element of infinite order, i.e., $G$ is not a torsion group. We define a homomorphism from $G$ to $X=\mathbb{T}$ as follows. By \cite[Theorems 23.1 and 24.1]{Fu}, $G$ can be embedded in $T \oplus \bigoplus_{\N}\Q$, where $T$ is a torsion group. Let $\iota$ be the embedding from $G$ into this direct sum. Let $g$ be an element of $G$ with infinite order. Then, there is a projection $\rho$ from $T \oplus \bigoplus_{\N}\Q$ onto $\Q$ such that $\rho(\iota(g))$ is nonzero. Choose an irrational $\alpha\in\T$.
Then the map $h\mapsto \rho(\iota(h))\alpha$ is a homomorphism from $G$ to $\T$ with dense image, so in this case, we will let $X=\T$ and let $\phi\colon G\to\T$ be such that $\phi(h)=\rho(\iota(h))\alpha$.

Next, assume that $G$ is a torsion group. For $p$ a prime and $n\in\N$, let $H_{p,n}$ be the Pr\"ufer $p$-group $\Z[p^\infty]$. (By definition, for $n_1,n_2\in\N$, $H_{p,n_1}\cong H_{p,n_2}$.)
By \cite[Theorems 23.1 and 24.1]{Fu}, $G$ can be embedded into $\bigoplus_{p\text{ prime}}\bigoplus_{n\in\mathbb{N}}H_{p,n}$.
Let $i$ be the corresponding inclusion map, and let $\pi_{p,n}$ be the projection from the direct sum onto $H_{p,n}$.
For convenience, we abbreviate $\rho_{p,n} = \pi_{p,n} \circ i$.
There are two cases.

First, assume there is some pair $(p,n)$ such that $|\rho_{p,n}(G)|$ is infinite. Note that $H_{p,n}$ is isomorphic to the subgroup $\{a/p^k\colon a\in\Z,k\in\N\}$ of $\T$; call this isomorphism $\psi$. Let $\phi=\psi\circ\rho_{p,n}$. Then $\phi\colon G\to\T$ is a homomorphism with dense image, so we set $X=\mathbb{T}$.

On the other hand, assume that for all pairs $(p,n)$, $|\rho_{p,n}(G)|$ is finite. Then $G$ is embedded in $\bigoplus_p\bigoplus_n\rho_{p,n}(G)$.
Since $G$ is a subgroup of a direct sum of finite cyclic groups, by \cite[Theorem 18.1]{Fu}, there is an infinite sequence $(n_j)_{j\in\N}$ of integers greater than $1$ such that $G\cong \bigoplus_{j}\Z/{n_j}\Z$.  We will assume without loss of generality that $n_{j+1}\geq n_j$ for all $j\in\N$.  Then there is a dense embedding $\phi$ from $G\cong \bigoplus_j \Z/{n_j}\Z$ to the  group  $X=\prod_{j}\Z/{n_j}\Z$.
 For each $n_j$, define a metric $\delta_j$ on $\Z/n_j\Z$ by $\delta_j([x],[y]) = \min_{k\in\Z} |x - y + kn_j|$, where $x,y\in\mathbb{Z}$ and $[x]$ and $[y]$ are the equivalence classes of $x$ and $y$ in $\Z/n_j\Z$. Observe that $\delta_j(x,y)\in \{0,1,\dots,\lfloor n_j/2\rfloor\}$ for all $x,y\in \Z/n_j\Z$. 
We can give $X$ the product metric $\delta(\mathbf{x},\mathbf{y})=\sum_{j\in\N}n_j^{-2j}\delta_j(x_j,y_j)$.
The Haar measure $\mu$ on $X$ is the product measure of the normalized counting measures on each $\Z/{n_j}\Z$. We will show that for any $a\in[0,1]$, the set $\{\mathbf{x}\in X\colon\delta(\mathbf{x},0)=a\}$ has measure zero. Since $\delta$ is invariant under the group operation of $X$, this will imply $\{\mathbf{x}\in X\colon\delta(\mathbf{x},\mathbf{y})=a\}$ has measure zero for all $\mathbf{y}\in X$.

Let $a\in [0,1]$; we first claim that there is at most one sequence $\mathbf{s}\in X$ with $s_j\in \{0,1,\dots,\lfloor n_j/2\rfloor\}$ for all $j\in\N$ and
\begin{equation}
\label{lksdajf}
\sum_{j\in\mathbb{N}}\frac{s_j}{n_j^{2j}}  = a.
\end{equation}
Assume there are two distinct such sequences $\mathbf{r},\mathbf{s}\in X$. Let $k$ be the smallest positive integer with $s_k\neq r_k$. Then 
\[\left|\sum_{j\leq k}\frac{s_j - r_j}{n_j^{2j}}\right|\geq \frac{1}{n_k^{2k}}.\]
However,
\[\left|\sum_{j> k}\frac{s_j - r_j}{n_j^{2j}}\right|\leq \sum_{j> k}\frac{|s_j - r_j|}{n_j^{2j}}\leq \sum_{j>k}\frac{1}{n_j^{2j-1}}\leq \sum_{j>k}\frac{1}{n_k^{2j-1}} = \frac{1}{n_k^{2k+1}} \frac{n_k^2}{n_k^2-1}<\frac{1}{n_k^{2k}},\]
so we cannot have 
\[\sum_{j\in\mathbb{N}}\frac{s_j}{n_j^{2j}}  = a = 
\sum_{j\in\mathbb{N}}\frac{r_j}{n_j^{2j}}.\]

Now assume $\mathbf{s}\in X$ satisfies \eqref{lksdajf}.
We will show 
\[\mu(\{\mathbf{x}\in X\colon\delta_j(x_j,0) = s_j\text{ for all }j\in \N\}) = 0.\]
Let $\mu_j$ be the normalized counting measure on $\Z/n_j\Z$, so $\mu_j(A) = |A|/n_j$ for any $A\subseteq\Z/n_j\Z$. If $n_j = 2$, then $\mu_j(\{x\in \Z/n_j\Z\colon\delta_j(x,0) = s_j\}) = 1/2<2/3$, and if $n_j > 2$, then $\mu_j(\{x\in \Z/n_j\Z\colon\delta_j(x,0)=s_j\}) \leq 2/n_j\leq 2/3$. Thus for all $k\in\mathbb{N}$,
\begin{align*}
\mu(\{\mathbf{x}\in X\colon\delta_j(x_j,0) = s_j\text{ for all }j\in \N\})&\leq \mu(\{\mathbf{x}\in X\colon\delta_j(x_j,0) = s_j\text{ for all }j\leq k\})\\&\leq (2/3)^k,
\end{align*}
so 
\[\mu(\{\mathbf{x}\in X\colon\delta_j(x_j,0) = s_j\text{ for all }j\in \N\}) = 0\]
and thus $\mu(\{\mathbf{x}\in X\colon\delta(\mathbf{x},0)=a\})=0$.
If there is no sequence $\mathbf{s}\in X$ with $s_j\in \{0,1,\dots,\lfloor n_j/2\rfloor\}$ satisfying \eqref{lksdajf}, we also have $\mu(\{\mathbf{x}\in X\colon\delta(\mathbf{x},0)=a\})=0$.

Here, the casework ends, and for $G$, the following is true. We have a  homomorphism with dense image   $\phi \colon G\to X$, where $X$ is a compact  abelian  group. Moreover, $X$ is endowed with an invariant metric $\delta$ such that for $x\in X$ and $a\in [ 0,1 ] $, we have $\mu(\{y\in X\colon\delta(x,y)=a\})=0$, where $\mu$ is the normalized Haar measure on $X$.  Additionally, $\delta(x,y)\leq 1$ for all $x,y\in X$. 

We define a natural $G$-action on $X$ by isometries via $gx=x+\phi(g)$. Since $\Img(\phi)$ is dense, the action is uniquely ergodic, with the normalized Haar measure $\mu$ as the unique invariant measure. Thus, by Theorem \ref{uniqueergodic}, for any $f\colon X\to\mathbb{R}$ for which the set of discontinuities has measure zero and $x\in X$, \begin{equation}\label{ergodicbusiness}
\lim_{n\to\infty}\frac{1}{|\Phi_n|}\sum_{g\in \Phi_n}f(gx)=\int_Xf\,d\mu.
\end{equation}
Suppose $Y\subseteq X$ is a Borel set with $\mu(\partial Y)=0$, and write $R_Y=\{g\in G\colon\phi(g)\in Y\}$.
Taking $x = 0$ and $f = \mathbf{1}_Y$, equation \eqref{ergodicbusiness} becomes \[d_\Phi(R_Y)=\lim_{n\to\infty}\frac{1}{|\Phi_n|}\sum_{g\in\Phi_n}\mathbf{1}_Y(\phi(g))=\int_X\mathbf{1}_Y\,d\mu=\mu(Y),\]
  since the set of discontinuities of $\mathbf{1}_Y$ is exactly $\partial Y$.

The remainder of the argument proceeds along the same lines as in Theorem \ref{ARall}.  Enumerate $G = \{g_1,g_2,\dots\}$,  let $t\in(0,1)$, and for $n\geq 1$, let  $r_n=t/(n\cdot 2^{ n-1})$. Let
\[B_{n}=\bigcup^{n }_{k=1}(\{x\in X\colon\delta(0,x)<r_{n}\}+\phi(g_{k})),\] and let $C_{n}=\bigcup^n_{k= 1 } B_{ k }$. Using Lemma \ref{fivepointone}, for each $n$, find a cofinite $R_{B_n}'\subset R_{B_n}$ such that  
\[d_{\Phi}{\left(\bigcup^\infty_{n= 1 } R_{B_n}'\right)}=\lim_{n\to\infty}\mu(C_{n})=\mu{\left(\bigcup^\infty_{n= 1 }B_{n}\right)}.\] Let $f_n(t)=\mu(C_{n})$. Note that since the construction of the sets $C_n$ depends continuously on $t$, the functions $f_n$ are continuous. Let $f=\lim_{n\to\infty}f_n$.
The convergence of $(f_n)$ to $f$ is uniform, so $f$ is continuous. Note that $\lim_{t\to 1}f(t)=1$  (since $\text{diam}\,(X)\leq 1$)  and $\lim_{t\to 0}f(t)=0$, so we can choose $t$ so that $d_{\Phi}(\bigcup R'_{B_n})=1-c$. Then, if $A=G\setminus\bigcup R'_{B_n}$, we have $d_{\Phi}(A)=c$. Let $n\in\N$. Then $(A-g)\cap\{h\in G\colon\delta(\phi(h),0)<r_{k}\}$ is finite, so $A-g_{k}$ cannot be a set of recurrence.
\end{proof}

Theorem \ref{BetterARall} can be extended to  countably infinite  commutative cancellative semigroups, allowing us to conclude that discordant subsets with any density in $(0,1)$ exist in, e.g., $(\N,+)$ and $(\N, \times)$ as well.

\begin{corollary}\label{semiAll}
Let $G$ be a  countably infinite  commutative cancellative semigroup, let $\Phi$ be a F\o lner sequence in $G$, and let $c\in (0,1)$. Then there is a discordant set $A \subseteq G$ with $d_{\Phi}(A)=c$. 
\end{corollary}

\begin{proof}
We can embed $G$ in a  countably infinite  abelian group $\widetilde{G}$ by a process similar to that of embedding an integral domain in a field (see \cite[Exercise 1.1.8]{HS} for more details about this construction). 
The elements of $\widetilde{G}$ can be thought formally as ``fractions'' of elements of $G$, so for any $g\in \widetilde{G}$, there is an $h\in G$ such that $gh\in G$.  
We will show that $\Phi$ is a F\o lner sequence in $\widetilde{G}$ (see \cite[Theorem 2.12]{BDM}). Let $g\in \widetilde{G}$, and let $h\in G$ be such that $hg\in G$. For each $n\in\N$,
\[\frac{|\Phi_n\bigtriangleup g\Phi_n|}{|\Phi_n|}\leq \frac{|\Phi_n\bigtriangleup hg\Phi_n|}{|\Phi_n|}+\frac{|g\Phi_n\bigtriangleup hg\Phi_n|}{|\Phi_n|}=\frac{|\Phi_n\bigtriangleup hg\Phi_n|}{|\Phi_n|}+\frac{|\Phi_n\bigtriangleup h\Phi_n|}{|\Phi_n|},\]
so 
\[\lim_{n\to\infty}\frac{|\Phi_n\bigtriangleup g\Phi_n|}{|\Phi_n|}=0,\]
and thus $\Phi$ is a F\o lner sequence in $\widetilde{G}$.
By Theorem \ref{BetterARall}, there is a discordant set $\widetilde{A}\subseteq \widetilde{G}$ with $d_\Phi(\widetilde{A})=c$. Set $A=\widetilde{A}\cap G$. We claim that $A$ is not piecewise syndetic in $G$. Suppose, toward the contrapositive, that there is a finite $H\subseteq G$ such that $H^{-1}A$ is thick. Let $F\subseteq \widetilde{G}$ be finite and select $g\in G$ such that $Fg\subseteq G$. Then there is an $h\in G$ such that $Fgh\subseteq H^{-1}A$. But then $H^{-1}A\subseteq H^{-1}\widetilde{A}$ is piecewise syndetic in $\widetilde{G}$, which establishes the claim.
Thus $A$ is a non-piecewise syndetic subset of $G$ with $d_\Phi(A)=c$. 
\end{proof}

\section{Large subsets of \texorpdfstring{$\{0,1\}^G$}{\{0,1\}\textasciicircum G}}\label{topology}

Let $G$ be a  countably infinite  semigroup.
In this section we will identify a subset $A\subseteq G$ with its indicator function $\mathbf1_A\in\{0,1\}^G$. We will say that $\mathbf1_A$ is \textit{thick}, \textit{syndetic}, etc.\ whenever $A$ possesses these properties. Thus, families of subsets of $G$ are identified with subsets of $\{0,1\}^G$. To study the topological properties of subsets of $\{0,1\}^G$, we give $\{0,1\}^G$ the product topology, making it a Cantor space. In this space, a base for the topology is given by the \textit{cylinder sets}, which have the form \[V(L_1,L_2)=\{\alpha\in\{0,1\}^G\colon \text{$\alpha(L_1)=\{1\}$ and $\alpha(L_2)=\{0\}$}\}\] for finite, disjoint $L_1,L_2\subseteq G$. Let us write \[\mathcal L=\{(L_1,L_2)\in\mathcal P(G)\times\mathcal P(G)\colon\text{$L_1$ and $L_2$ are finite and disjoint}\}.\]

\subsection{Background and motivation}

Recall that if $X$ is a topological space homeomorphic to the Cantor set, it is a \textit{Baire space}, meaning that it cannot be written as a countable union of nowhere dense sets. A subset of $X$ which is a countable union of nowhere dense sets is called \textit{meager}, and a subset of $X$ whose complement is meager is \textit{comeager}. The fact that $X$ is a Baire space means that no subset can be both meager and comeager, so one can think of the elements of a comeager set as being ``topologically generic''. We wish to stress in this section the analogy between topology and measure theory in which ``meager'' corresponds to ``measure 0'' and ``comeager'' corresponds to ``measure 1''.

We will describe comeager subsets of $\{0,1\}^G$, whose elements can thus be thought of as ``topologically generic'' subsets of $G$. Let us first motivate our results by examining them in a familiar setting. Consider the case $G=\N$, where $\{0,1\}^\N$ is the space of infinite binary sequences. An infinite binary sequence is called \textit{disjunctive} if every finite binary word occurs within it. (Formally, $\alpha\in\{0,1\}^\N$ is disjunctive if, for each $k\in\N$ and $\omega\in\{0,1\}^k$, there exists $n\in\N$ such that $\alpha(n+i-1)=\omega(i)$ for each $1\leq i\leq k$.) The subset of $\{0,1\}^\N$ consisting of disjunctive sequences is large both topologically and measure-theoretically: it is comeager and of measure 1.

However, the topological and measure-theoretical worlds diverge once we consider the limiting frequencies with which finite subwords occur. Given a sequence $\alpha\in\{0,1\}^\N$ and a finite word $\omega\in\{0,1\}^k$, one says that $\omega$ occurs in $\alpha$ with a limiting frequency \begin{equation}\label{normalnormal}
\lim_{n\to\infty}\frac{|\{1\leq i\leq n-k+1\colon\text{$\alpha(i+j-1)=\omega(j)$ for each $1\leq j\leq k$}\}|}{n}.
\end{equation}
It is a classical result of \'Emile Borel \cite{Bo} that the set of \textit{normal} sequences, those for which every word of length $k$ occurs as a subword with limiting frequency $2^{-k}$, has measure 1. Moreover, as we will see, it is a meager set.

On the other hand, one can consider sequences for which limits analogous to \eqref{normalnormal} (but modified slightly to avoid counting overlapping sequences) fail to exist as badly as possible. For these sequences, the $\limsup$ of the frequency of any finite word is 1 and the corresponding $\liminf$ is 0. We will call such sequences \textit{extremely non-averageable} (or \textit{ENA}). It turns out that the set of ENA sequences, in contrast to the set of normal sequences, is comeager \cite[Theorem 1]{CZ}.

We will make the following generalizations of the facts we have just stated: assuming only cancellativity, we prove an analogue in a  countably infinite  semigroup $G$ to the statement that the set of disjunctive sequences is comeager (Theorem \ref{disjunctive}); assuming also amenability, we prove an analogue to the statement that the set of ENA sequences is comeager (Theorem \ref{topol-normal}). A generalization of measure-theoretical normality to  countably infinite  amenable semigroups and the corresponding results can be found in \cite{BDM}. 

Since the set of normal sequences and the set of ENA sequences are disjoint, we see also that the set normal sequences is meager, and that the set of ENA sequences has measure 0 (see \cite[Proposition 4.7]{BDM} regarding the former statement in arbitrary amenable cancellative semigroups). This allows us to give another motivation for the results of this section. One reason we gave for studying discordant sets is that they illustrate how topological and analytic notions of largeness for subsets of $G$ diverge; in this section, we construct subsets of $\{0,1\}^G$ on which basic topological and analytic notions of largeness disagree. 

Finally, we remark that, although we prove the results of this section in the space $\{0,1\}^G$, their generalizations to $\{0,\dots,n-1\}^G$ (which can be thought of as the space of $n$-part partitions, or $n$-colorings, of $G$) are apparent and can be achieved with the same arguments.

\subsection{Disjunctive elements of $\{0,1\}^G$}

In this subsection we will assume that $G$ is cancellative. Call $\alpha\in\{0,1\}^G$ \textit{disjunctive} if for any $(L_1,L_2)\in\mathcal L$ there is some $g\in G$ for which $\alpha\in V(L_1g,L_2g)$.\footnote{In the notation of Section \ref{topol-dynamics}, we can equivalently say that $\alpha$ is disjunctive if and only if, for every $(L_1,L_2) \in \mathcal{L}$, there exists $g \in G$ for which $g\alpha \in V(L_1,L_2)$. This definition has the benefit of making sense for non-cancellative semigroups.}
Cancellativity here guarantees that $L_1g\cap L_2g=\varnothing$ for all $g\in G$. Note that, for $G=\N$, this definition specializes to the definition of disjunctive sequences we gave above.

\begin{theorem}\label{disjunctive}
The set of disjunctive elements of $\{0,1\}^G$ is comeager.
\end{theorem}

\begin{proof}
For any $(L_1,L_2)\in\mathcal L$, let $A_{L_1,L_2}=\bigcup_{g\in G}V(L_1g,L_2g)$, so that the set of disjunctive elements of $\{0,1\}^G$ is \[A=\bigcap_{(L_1,L_2)\in\mathcal L}A_{L_1,L_2}.\] Now let $(L_1,L_2),(M_1,M_2)\in\mathcal L$. Since $G$ is cancellative and infinite, we can find some $g\in G$ for which $(L_1g\cup L_2g)\cap(M_1\cup M_2)=\varnothing$. Then we have $\mathbf1_{L_1g\cup M_1}\in A_{L_1,L_2}\cap V(M_1,M_2),$ so that $A_{L_1,L_2}$ is dense, since it has nontrivial intersection with each element of the open base for $\{0,1\}^G$. Since $A_{L_1,L_2}$ is a union of open sets, we have written $A$ as a countable intersection of open dense sets, which proves that $A$ is comeager.
\end{proof}

We deduce as a relevant consequence of this theorem that the set of non-piecewise syndetic elements of $\{0,1\}^G$ is meager, since any disjunctive element of $\{0,1\}^G$ is thick. If $G$ is non-cancellative, the set of piecewise syndetic elements of $\{0,1\}^G$ may fail to be comeager. For example, if there exists $z\in G$ such that for all $g\in G$, one has $gz=z$, then $\alpha\in\{0,1\}^G$ is piecewise syndetic if and only if $\alpha(z)=1$, so the family of piecewise syndetic sets cannot be dense.

\subsection{Extremely non-averageable elements of $\{0,1\}^G$}

In this subsection we will assume that $G$ is both cancellative and amenable.
Fix a F\o lner sequence $\Phi$ in $G$.
We will begin by defining $\Phi$-normality of elements of $\{0,1\}^G$, as defined in \cite{BDM}, for comparison with the definition we will give for ENA.
We will also utilize the notion of $\Phi$-normality in Theorem \ref{normalena}.

Given $\alpha\in\{0,1\}^G$, one says that $\alpha$ is \textit{$\Phi$-normal} if, for each finite $K\subseteq G$ and $\omega\in\{0,1\}^K$, \[\lim_{n\to\infty}\frac{|\{g\in G\colon\text{$Kg\subseteq\Phi_n$ and $\alpha(hg)=\omega(h)$ for each $h\in K$}\}|}{|\Phi_n|}=\frac{1}{2^{|K|}}.\]
The following classical lemma regarding F\o lner sequences, which states that they have a property analogous to thickness, assures us that the notion of $\Phi$-normality is meaningful.

\begin{lemma}\label{topnormal-nondegeneracy}
Let $\Phi$ be a F\o lner sequence in a  countably infinite  cancellative semigroup $G$. Let $K\subseteq G$ be finite. Then there exists $N\in\N$ such that, for each $n\geq N$, $Kg_n\subseteq\Phi_n$ for some $g_n\in G$.
\end{lemma}

\begin{proof}
Find $N$ large enough that, for each $n\geq N$ and $h\in K$, \[1-\frac{1}{|K|}<\frac{|\Phi_n\cap h\Phi_n|}{|\Phi_n|}=\frac{|h(h^{-1}\Phi_n\cap\Phi_n)|}{|\Phi_n|}=\frac{|h^{-1}\Phi_n\cap\Phi_n|}{|\Phi_n|}.\] Then $\bigcap_{h\in K}h^{-1}\Phi_n\neq\varnothing$, so let $g_n$ be an element of this set. Then $Kg_n\subseteq\Phi_n$.
\end{proof}

Now we define extreme non-averageability in $\{0,1\}^G$.
The following definitions are admittedly unwieldy because we wish to measure the frequency of occurrences \textit{without overlap} of our finite sequence.
This is because we are interested in the maximal frequency with which the finite sequence can occur, and when allowing for overlaps, this frequency is difficult to predict.
We would be pleased if a simpler equivalent definition of ENA in an arbitrary  countably infinite  amenable cancellative semigroup could be found.

Fix a finite $K\subseteq G$. Given a finite set $X\subseteq G$, define \[\mathcal Y_{K,X}=\{Y\subseteq X\colon\text{$KY\subseteq X$, $Ky\cap Kz=\varnothing$ for all $y,z\in Y$, and $|Y|$ is maximal}\}.\] Now let $\omega\in\{0,1\}^K$, and $\alpha\in\{0,1\}^G$, and define \[\overline d_\Phi(K,\omega,\alpha)=\limsup_{n\to\infty}\max_{Y\in\mathcal Y_{K,\Phi_n}}\frac{|\{y\in Y\colon\text{$\alpha(hy)=\omega(h)$ for all $h\in K$}\}|}{|Y|}.\]
We observe that $\overline d_\Phi(K,\omega,\alpha)$ is well-defined since $\mathcal Y_{K,\Phi_n}=\{\varnothing\}$ for only finitely many $n$ by Lemma \ref{topnormal-nondegeneracy}. Define $\underline d_\Phi(K,\omega,\alpha)$ to be the corresponding $\liminf$. Hence, $\overline d_\Phi(K,\omega,\alpha)$ and $\underline d_\Phi(K,\omega,\alpha)$ give the (upper and lower) limiting frequencies, measured along $\Phi$, with which the finite binary word $\omega$ occurs without overlaps in $\alpha$. For convenience, we write \[\mathcal K=\{(K,\omega)\colon\text{$K\subseteq G$ is finite and $\omega\in\{0,1\}^K$}\}.\] We say that $\alpha$ is \textit{$\Phi$-ENA} if $\overline d_\Phi(K,\omega,\alpha)=1$ and $\underline d_\Phi(K,\omega,\alpha)=0$ for each $(K,\omega)\in\mathcal K$.

With these definitions in hand, we proceed to the main result of this subsection.

\begin{theorem}\label{topol-normal}
The set of $\Phi$-ENA elements of $\{0,1\}^G$ is comeager.
\end{theorem}

\begin{proof}
For each $(K,\omega)\in\mathcal K$, let \[A_{K,\omega}=\{\alpha\in\{0,1\}^G\colon\text{$\overline d_\Phi(K,\omega,\alpha)=1$ and $\underline d_\Phi(K,\omega,\alpha)=0$}\},\] so that the set of $\Phi$-ENA elements of $\{0,1\}^G$ is $\bigcap_{(K,\omega)\in\mathcal K}A_{K,\omega}$. Fix $(K,\omega)\in\mathcal K$; it suffices to show that $A_{K,\omega}$ is comeager.

Find a subsequence $\Phi'=(\Phi_{n_k})_{k\in\N}$ of $\Phi$ which satisfies \[\lim_{k\to\infty}\frac{|\Phi_{n_1}|+\cdots+|\Phi_{n_{k-1}}|}{|\Phi_{n_k}|}=0.\] Define $\Psi_1=\Phi_1$ and, for $k\in\N$, $\Psi_{k+1}=\Phi_{n_{k+1}}\setminus\bigcup_{m=1}^k\Phi_{n_m}$; let $\Psi$ denote the F{\o}lner sequence $(\Psi_n)_{n\in\N}$. Fix $g\in K$. Now for each $n\in\N$ choose $Y_n\in\mathcal Y_{K,\Psi_n}$ and define
\begin{align*}
B_n^+&=\{\beta\in\{0,1\}^G\colon\text{for all $y\in Y_n$ and $h\in K$, $\beta(hy)=\omega(h)$}\},\\
B_n^-&=\{\beta\in\{0,1\}^G\colon\text{for all $y\in \Psi_n$ and $h\in K$, $\beta(hy)\neq\omega(g)$}\}.
\end{align*}
The definition of $\mathcal Y_{K,\Psi_n}$ ensures that the sets $B_n^\pm$ are nonempty for all large enough $n$, and they are clearly open since they are cylinder sets. In particular, $B_n^+=V(\omega^{-1}(\{1\})Y_n,\omega^{-1}(\{0\})Y_n)$, and either $B_n^-=V(\varnothing, K\Psi_n)$ or $B_n^-=V(K\Psi_n,\varnothing)$, depending on whether $\omega(g)$ is $1$ or $0$. The set $B_n^-$ is chosen this way to ensure that for all $Y_n\in \mathcal{Y}_{K,\Psi_n}$, there is no $y\in Y_n$ such that $\beta(hy)=\omega(h)$ for all $h\in K$. In particular, $\beta(gy)\neq \omega(g)$ for all $y$. Finally, define $C_n^\pm=\bigcup_{m=n}^\infty B_m^\pm$ for all $n\in\N$.

Let $(L_1,L_2)\in\mathcal L$. Since the sets $\Psi_n$ are disjoint and $L_1$ and $L_2$ are finite, cancellativity gives that for all large enough $n\in\mathbb{N}$, $B_{ n}^\pm\cap V(L_1,L_2)\neq\varnothing$ and hence $C_n^\pm\cap V(L_1,L_2)\neq\varnothing$.
It follows that $C_n^\pm$ is open and dense. Define $C=\bigcap(C_n^+\cap C_n^-)$, so that $C$ is a comeager subset of $\{0,1\}^G$. Moreover, if $\gamma\in C$, then
\begin{equation}\label{annd}
\overline d_\Psi(K,\omega,\gamma)=1\qquad\text{and}\qquad\underline d_\Psi(K,\omega,\gamma)=0
\end{equation}
by construction. Since $\lim_{k\to\infty}|\Phi_{n_k}|/|\Psi_k|=1$, \eqref{annd} must hold with respect to $\Phi'$ in addition to $\Psi$, and thus also with respect to $\Phi$. Hence $C$ is a comeager subset of $A_{K,\omega}$, as we had hoped.
\end{proof}

It is important to note that the property of $\Phi$-ENA, as well as $\Phi$-normality, is dependent on $\Phi$. The following theorem shows that this is true in a strong sense.

\begin{theorem}\label{normalena}
Let $\alpha\in\{0,1\}^G$ be disjunctive. Then there exist F{\o}lner sequences $\Psi^1$ and $\Psi^2$ in $G$ such that $\alpha$ is simultaneously $\Psi^1$-normal and $\Psi^2$-ENA.
\end{theorem}

\begin{proof}
Let $\Phi$ be a F\o lner sequence of $G$, and suppose  $\beta_1,\beta_2\in\{0,1\}^G$ are $\Phi$-normal and $\Phi$-ENA, respectively. For each $n\in\N$ and $i\in\{1,2\}$, define $X_n^i=\Phi_n\cap\beta_i^{-1}(\{1\})$ and $Y_n^i=\Phi_n\cap\beta_i^{-1}(\{0\})$ and find $g_n^i\in G$ such that $\alpha\in V(X_n^ig_n^i,Y_n^ig_n^i)$. Then define the F{\o}lner sequences $\Psi^i=(\Psi_n^i)_{n\in\N}$ by $\Psi_n^i=\Phi_ng_n^i$. These F{\o}lner sequences have the desired properties.
\end{proof}

\section{Topological dynamics and piecewise syndeticity}\label{topol-dynamics}

Let $G$ denote any  countably infinite  (not necessarily cancellative) semigroup.
In this section we explore piecewise syndeticity of subsets of $G$ in terms of dynamical properties of elements of $\{0, 1\}^G$ and their orbit closures.
Some of the results presented below are well-known for subsets of $\N$.
Using these results we revisit and examine in greater detail the comparison between van der Waerden's theorem and Szemer\'edi's theorem made in the introduction.

We adopt the notation of the previous section.
For each $g\in G$, let $\rho_g\colon G\to G$ be the map $h\mapsto hg$. We will consider the dynamical system $(\{0,1\}^G,G)$, where $G$ acts on $\{0,1\}^G$ by $g\alpha=\alpha\circ\rho_g$.
In the familiar settings $G=\N$ and $G=\Z$, this is just the left-shift action on infinite one- or two-sided sequences.
In light of the natural correspondence between $\{0,1\}^G$ and $\mathcal{P}(G)$, we may view this dynamical system instead as $(\mathcal{P}(G), G)$. In this case, $g$ acts on $\mathcal{P}(G)$ by $A \mapsto Ag^{-1}$.
This action is continuous, since for any $(L_1,L_2)\in\mathcal L$, we have $g^{-1}V(L_1,L_2)=V(L_1g,L_2g)$  (for this identity to hold in the non-cancellative case, we adopt the convention that $V(A,B) = \varnothing$ if $A\cap B\neq \varnothing$, in view of the fact that $g^{-1}V(L_1,L_2) = \varnothing$ if and only if  $L_1g\cap L_2g \neq \varnothing$) .

Given $\alpha\in\{0,1\}^G$, we define the \textit{orbit closure} of $\alpha$ to be
\[
\mathcal O(\alpha)=\overline{\{g\alpha\colon g\in G\}}.
\]
Now fix an increasing sequence $(F_n)$ of finite subsets of $G$ such that $\bigcup F_n=G$.
Let us note that for any $\alpha\in\{0,1\}^G$, we have $\beta\in\mathcal O(\alpha)$ if and only if, for all $n\in\N$, there exists $g\in G$ such that $\beta(x)=(g\alpha)(x)$ for each $x\in F_n$.
This simple but crucial property of orbit closures in the system $(\{0,1\}^G, G)$ follows immediately from the definition of the product topology on $\{0,1\}^G$.

We illustrate what this means in the familiar case $G=\N$. In this setting, $\beta \in \mathcal{O}(\alpha)$ if and only if, for any $n\in\N$, $\alpha$ can be shifted to agree with $\beta$ on $\{1,\dots,n\}$.
Put another way, $\beta \in \mathcal{O}(\alpha)$ means that one can find arbitrarily long initial segments of $\beta$ as subwords of $\alpha$.

Before stating the results, we give convenient reformulations of the notions of syndeticity, piecewise syndeticity, and thickness in terms of indicator functions. An element $\alpha\in\{0,1\}^G$ is syndetic if and only if there exists a finite set $H\subseteq G$ such that, for each $x\in G$, we have $\alpha(hx)=1$ for some $h\in H$.
Similarly, $\alpha$ is piecewise syndetic if and only if there exists a finite $H\subseteq G$ and an infinite sequence $(g_i)$ in $G$ such that, for each $n\in\N$ and $x\in F_n$, we have $\alpha(hxg_n)=1$ for some $h\in H$.
Finally, $\alpha$ is thick if and only if $\mathbf1_G \in \mathcal{O}(\alpha)$, so by Lemma \ref{ST-duality}, $\alpha$ is syndetic if and only if $\mathbf1_\varnothing \notin \mathcal{O}(\alpha)$.

\begin{lemma}\label{psdynamics} An element $\alpha\in\{0,1\}^G$ is piecewise syndetic if and only if the orbit closure $\mathcal O(\alpha)$ contains a syndetic element.
\end{lemma}

\begin{proof}
Suppose $\alpha\in\{0,1\}^G$ is piecewise syndetic and find a finite set $H\subseteq G$ and an infinite sequence $(g_n)$ in $G$ which witnesses $\alpha$'s piecewise syndeticity, in the sense of the previous paragraph.
We claim that there is a decreasing sequence $(N_m)$ of infinite subsets of $\N$ with the property that $\alpha(xg_{n_1})=\alpha(xg_{n_2})$ for each $x\in F_m$ and $n_1, n_2\in N_m$.
Indeed, for any $m$, there are finitely many functions $F_m\to\{0,1\}$, so such a sequence can easily be constructed inductively by applying the pigeonhole principle.
Hence we can define $\beta\in\{0,1\}^G$ by $\beta(x)=\alpha(xg_n)=(g_n\alpha)(x)$ for any $x\in F_m$ and $n\in N_m$.
By construction, $\beta\in\mathcal O(\alpha)$. To show that $\beta$ is syndetic, let $x\in G$, find $m$ such that $x\in F_m$, and find $n\in N_m$ with $n\geq m$. Since $x\in F_n$, we can find $h\in H$ such that $\alpha(hxg_n)=1$.
Then $\beta(hx)=\alpha(hxg_n)=1$, so $\beta$ is syndetic.

Now suppose $\alpha\in\{0,1\}^G$ and $\beta\in\mathcal O(\alpha)$ is syndetic. Let $H\subseteq G$ be a finite set which witnesses $\beta$'s syndeticity. Fix $n\in\N$ and find $m\in\N$ such that $HF_n\subseteq F_m$.
Find $g_m\in G$ such that $\beta(y)=(g_m\alpha)(y)$ for each $y\in F_m$.
Let $x\in F_n$ and find $h\in H$ for which $\beta(hx)=1$.
Then $hx\in F_m$, so $\alpha(hxg_m)=\beta(hx)=1$.
This proves that $\alpha$ is piecewise syndetic.
\end{proof}

% Using this dynamical result, we prove an equivalent form of van der Waerden's theorem.

% \begin{theorem}[Van der Waerden's theorem]
% Suppose $A\subseteq\N$ is piecewise syndetic. Then $A$ is AP-rich.
% \end{theorem}

% \begin{proof}
% Let $\alpha=\mathbf1_A$. First, we note the easy fact that if $\mathcal O(\alpha)$ contains an AP-rich sequence, then $A$ must be AP-rich. By Lemma \ref{psdynamics}, find $\beta\in\mathcal O(\alpha)$ which is syndetic, and find $s\geq0$ such that, for each $n\in\N$, there exists $t\in\{n,\dots,n+s\}$ with $\beta(t)=1$.

% Give $\{0,1\}^\N$ the metric defined by $d(\chi,\omega)=1/(n+1)$ if $\chi(k)=\omega(k)$ for each $k\in\{1,\dots,n\}$ and $\chi(k+1)\neq\omega(k+1)$, and $d(\chi,\omega)=1$ if $\chi(1)\neq\omega(1)$. Letting $T$ denote the left-shift operator on $\{0,1\}^\N$, apply Lemma \ref{dynamicalvdw} to find $m,d\in\N$ such that the sequences $T^m\beta,T^{m+d}\beta,\dots,T^{m+(\ell-1) d}\beta$ are within $1/(s+1)$ of each other. That is, for each $t\in\{1,\dots,s\}$, we have \[\beta(t+m)=\beta(t+m+d)=\cdots=\beta(t+m+(\ell-1)d).\] Since some $t$ satisfies $\beta(t+m)=1$, it follows that $\beta$ has a length-$\ell$ arithmetic progression. Since $\ell$ was arbitrary, $\beta$ is AP-rich, which means that $A$ must be AP-rich.
% \end{proof}

% Since the non-piecewise syndetic sequences are exactly those whose orbit closures lack syndetic sets, we see in a different way that discordant subsets of $\N$ are those for which make Szemer\'edi's theorem harder to prove than van der Waerden's theorem.

Recall that \textit{subsystem} of a dynamical system $(X,G)$ is a nonempty closed $G$-invariant subset of $X$. A \textit{minimal subsystem} is a subsystem which is itself a minimal dynamical system.
By Zorn's lemma, any dynamical system has a minimal subsystem (see, e.g., \cite[Theorem 2.22]{GH}).
We will call the subsystem of $(\{0,1\}^G,G)$ consisting of the fixed point $\mathbf1_\varnothing$ (the zero function) the \textit{trivial subsystem}.

The following theorem is a restatement of Lemma \ref{psdynamics} in terms of the minimal subsystems of orbit closures.

\begin{theorem}\label{pssubsystems}
An element $\alpha\in\{0,1\}^G$ is piecewise syndetic if and only if $\mathcal O(\alpha)$ possesses a nontrivial minimal subsystem.
\end{theorem}

In particular, if $\alpha$ is discordant, $\mathcal{O}(\alpha)$ does not have a nontrivial minimal subsystem. Moreover, sets of positive upper density having this property are exactly the discordant sets.

\begin{proof}
Suppose $\alpha\in\{0,1\}^G$ is piecewise syndetic. By Lemma \ref{psdynamics}, there exists $\beta\in\mathcal O(\alpha)$ which is syndetic. Then $\mathcal O(\beta)\subseteq\mathcal O(\alpha)$ and $\mathbf1_\varnothing\notin\mathcal O(\beta)$, so $\mathcal O(\beta)$ has a nontrivial minimal subsystem.
It follows that $\mathcal{O}(\alpha)$ has a minimal subsystem as well.

Conversely, suppose $\mathcal O(\alpha)$ possesses a nontrivial minimal subsystem $X$. Then $\mathbf1_\varnothing\notin X$, so every element of $X$ is syndetic. Thus, by Lemma \ref{psdynamics}, $\alpha$ is piecewise syndetic.
\end{proof}

The existence or nonexistence of nontrivial minimal subsystems in $\mathcal{O}(\alpha)$ has significant combinatorial consequences, which we presently examine.

Let us specialize to the case $G = \N$.
As stated in the introduction, there is an equivalent formulation of van der Waerden's theorem as a topological recurrence theorem:

% \cite[Section 2]{Fur}.
% ergodic theorems and diophantine problems, quote both
% furstenberg weiss
% furstenberg poincare
\begin{theorem}[Topological van der Waerden's theorem]\label{dynamicalvdw}
Let $(X,T)$ be a minimal\footnote
{In the sources, one will not see the condition of minimality explicitly imposed in the formulations of Theorem \ref{dynamicalvdw}.
We of course do not lose generality by making this assumption as we may always pass to a minimal subsystem.
Moreover, the topological proofs of this theorem of which the authors are aware, with the notable exception of \cite[Proposition L]{BL}, use minimality in a crucial way.
}
topological dynamical system where $X$ is a compact metric space.
Then for any $\ell\in\N$ and $\varepsilon>0$ there exists $d\in\N$ and $x \in X$ such that the points $x,T^dx,\dots,T^{(\ell-1)d}x$ are within $\varepsilon$ of each other.
\end{theorem}

\begin{proof}
This is a special case of \cite[Theorem 1.4]{FW}.
See also \cite[Theorem 2.6]{Fur} or \cite[Corollary 3.9]{Be2}.
\end{proof}

From here the classical van der Waerden's theorem can be derived along the following lines.
Suppose $\alpha \in \{0,1\}^\N$ is piecewise syndetic.
Then $\mathcal{O}(\alpha)$ possesses a nontrivial minimal subsystem $Y$ by Theorem \ref{pssubsystems}.
Since $\mathbf1_\varnothing \notin Y$, every element of $Y$ is syndetic.
Giving $\{0,1\}^\N$ the product metric described in Subsection \ref{cantor} and applying Theorem \ref{dynamicalvdw} inside $(Y,T)$ shows that, for any $\ell \in \N$, the system $Y$ possesses an element with a length-$\ell$ arithmetic progression.
 Since $Y\subset\mathcal{O}(\alpha)$, for any $\ell\in\mathbb{N}$, there is an element of $\mathcal{O}(\alpha)$ with a length-$\ell$ arithmetic progression.  It follows that $\alpha$ is AP-rich.
(As we saw in the introduction, this is equivalent, via the partition regularity of piecewise syndetic sets, to the formulation of van der Waerden's theorem in terms of partitions of $\N$.)

This discussion supports our position taken in the introduction on the distinction between van der Waerden's and Szemer\'edi's theorems.
Specifically, if the orbit closure of any element of $\{0,1\}^\N$ having positive upper density were to possess a nontrivial minimal subsystem, then Szemer\'edi's theorem would follow immediately from the topological version of van der Waerden's theorem using the argument we gave in the preceding paragraph.

As our next combinatorial application of Theorem \ref{pssubsystems}, we will prove the partition regularity of piecewise syndetic sets from a dynamical perspective (cf. Theorem \ref{brown}).
We first address the following question: For which $\alpha\in\{0,1\}^G$ is $\mathcal O(\alpha)$ itself minimal? The answer is that every finite subword occurring in $\alpha$ must occur syndetically in $\alpha$, as the following lemma states. The result is well known in the case when $G$ is $\Z$ or $\N$; see, e.g., \cite[Proposition 1.22]{Fur}.

\begin{lemma}\label{minimalorbit}
Let $\alpha\in\{0,1\}^G$.
Then $\mathcal O(\alpha)$ is minimal if and only if, for all $(L_1,L_2)\in\mathcal L$, the set $\{g\in G\colon g\alpha\in V(L_1,L_2)\}$ is either empty or syndetic.
\end{lemma}

\begin{proof}
This is a consequence of the following general phenomenon: Let $(X,G)$ be a topological dynamical system where $X$ is a metric space, and let $x\in X$. Then $\mathcal O(x)=\overline{\{gx\colon g\in G\}}$ is minimal if and only if, for each open $V\subseteq X$, $R_V(x)$ is either empty or syndetic. Clearly, if $\mathcal O(x)$ is minimal, then this condition holds (see the proof of Lemma \ref{not-nowhere-dense}).

Conversely, suppose $y\in\mathcal O(x)$. We will show that $x\in\mathcal O(y)$, which will prove that $\mathcal O(x)$ is minimal. To this end, fix $n\in\N$ and find a finite set $H\subseteq G$ such that $H^{-1}R_{B(1/(2n),x)}(x)=G$. Find a sufficiently small $\varepsilon>0$ so  that $hB(\varepsilon,y)\subseteq B(1/(2n),hy)$ for each $h\in H$. Next find $g\in G$ for which $gx\in B(\varepsilon,y)$ and find $h\in H$ satisfying $hg\in R_{B(1/(2n),x)}(x)$. Then $hgx\in B(1/(2n),hy)\cap B(1/(2n),x)$, i.e., $d(hy,x)<1/n$. This proves that $x\in\mathcal O(y)$, as claimed.
\end{proof}

Before continuing, we remark that, as in the previous section, the preceding can be easily generalized to statements about $\{0,1,\dots,r\}^G$.
In this setting, we say that $\alpha\in\{0,1,\dots,r\}^G$ is (piecewise) syndetic if $\mathbf1_{\alpha^{-1}(\{1,\dots,r\})}\in\{0,1\}^G$ is (piecewise) syndetic in the usual sense.
We continue to refer to $\{\mathbf1_\varnothing\}$ as the trivial subsystem.
We leave the proofs of the generalizations to the interested reader.
Using this, we give the promised result on the partition regularity of piecewise syndetic sets.
The following proof in the case $G=\Z$ can be found in \cite[Theorem 1.24]{Fur}.

\begin{theorem}\label{brown2}
Let $A\subseteq G$ be piecewise syndetic and suppose it is partitioned as $A=\bigcup_{i=1}^rC_i$. Then $C_i$ is piecewise syndetic for some $i$.
\end{theorem}

\begin{proof}
Define $\alpha\in\{0,1,\dots,r\}^G$ by $\alpha(g)=0$ if $g
\notin A$ and $\alpha(g)=i$ if $g\in C_i$. By Lemma \ref{pssubsystems}, $\mathcal O(\alpha)$ has a nontrivial minimal subsystem $X$. Let $\beta\in X$. By Lemma \ref{minimalorbit} and the nontriviality of $X$,  for all $i\in\text{Range}\,\beta$ the set $\{g\in G\colon\beta(g)=i\}$ is syndetic. Let $i\in\text{Range}\,\beta$.  Define $\beta'\in\{0,1\}^G$ by $\beta'(g)=1$ if $\beta(g)=i$ and $\beta'(g)=0$ otherwise. Then $\beta'$ is a syndetic element in $\mathcal O(\mathbf1_{C_i})$, so $C_i$ must be piecewise syndetic.
\end{proof}

It is worth mentioning that Theorem \ref{brown2} is easily proved with the help of topological algebra in $\beta G$, the Stone--\v Cech compactification of $G$, that is, the space of ultrafilters on $G$.
It is well known that $A \subseteq G$ is piecewise syndetic if and only if some shift $g^{-1}A$ lies in a minimal idempotent of $\beta G$ \cite[Theorem 4.43]{HS}.
From here partition regularity follows from the definition of an ultrafilter.
This approach to Ramsey theory---learning about subsets of $G$ by studying $\beta G$---has proved very fruitful, leading to, for example, a slick proof of Hindman's theorem (see \cite[Corollary 5.9]{HS} or \cite[Theorem 1.2]{Be}).
We do not pursue this further but direct the curious reader towards the references we have just mentioned.

\printbibliography

\end{document}